\documentclass{amsart}

\usepackage{amsthm,amsfonts,amsmath,amssymb,latexsym,epsfig}
\usepackage{upref,amssymb,eucal,ae,enumitem,wasysym}
\usepackage[all,cmtip]{xy}

\newtheorem{theorem}{Theorem}

\newtheorem{corollary}[theorem]{Corollary}

\newenvironment{remark}{\medskip \refstepcounter{theorem}
\noindent  {\bf Remark \thetheorem}.\rm}{\,}

\def\rk3{\rm K3}

\def\<{\langle}
\def\>{\rangle}

\def\tn{\tilde{n}}
\def\tg{\tilde{g}}
\def\mb#1{{\mathbb #1}}
\def\mc#1{{\mathcal #1}}
\def\mf#1{{\mathfrak #1}}
\def\BOne{{\mathchoice {\rm 1\mskip-4mu l} {\rm 1\mskip-4mu l}
                          {\rm 1\mskip-4.5mu l} {\rm 1\mskip-5mu l}}}

\begin{document}
\title[Manifolds, model geometries, and volume related invariants]
{Closed manifolds, model geometries, and volume related differentiable 
invariants}

\begin{abstract}
By analyzing smooth paths of metrics on $M^n$ through their associated paths of
isotopic isometric embeddings into a standard sphere of large dimension, and 
the homotopy lift by Yamabe metrics of the associated path of conformal 
classes, we prove that if $M$ carries a metric $g$ of constant scalar 
curvature $s_g$ and Ricci tensor $r_g \leq 0$, and if $M$ does not carry 
scalar flat metrics other than Ricci flat ones, then $M$ 
is not a manifold of Kazdan-Warner type I, and if the space of Ricci flat 
metrics is not empty, $M$ is a manifold of Kazdan-Warner type II, while 
otherwise, $M$ is of Kazdan-Warner type III. 
If $M$ has a contractible universal cover and carries no Ricci flat 
metrics at all, $M$ is of Kazdan-Warner type III. 
If $M$ carries a metric $g'$ of 
nontrivial scalar curvature  $s_{g'}\geq 0$, and an Einstein metric $g_{-}$ 
such that $r_{g_{-}}<0$, then $M$ must carry both, scalar flat non Ricci flat 
and Ricci flat metrics, and if orientable, it is spinnable. No such manifold 
exists if $n\leq 3$, and if the Ricci tensor $r_{g'}$ is assumed further to be 
positive, no such manifold exists if $n\leq 4$, and in these dimensions, 
$M^{n}$ can admit Einstein metrics of scalar curvature of at most one sign.
If $[0,1]\ni t \rightarrow f_{g_t}:(M,g_t) \rightarrow (
\mb{S}^{\tn},\tg)$ is a smooth path of {\it differentiable} isometric 
embeddings of a smooth path $[0,1]\ni t \rightarrow g_{t}$ of differentiable
Yamabe metrics on $M$,  
the sectional curvature of $(f_{g_t}(M),g_t)$ is constant at some $t$ 
if, and only if, it is constant for all $t$, and the sign of the constant
functions is unique. 
A manifold $M$ with a model geometry 
$({\rm Isom}(X,g),X)$ is of the form $M=X/\Gamma_M$ where $\Gamma_M$ is a
discrete subgroup of ${\rm Isom}(X,g)$ that acts without fixed points, and has 
its differentiable structure and a distinguished metric $g_M$  
determined by the model $(X,g)$ and fundamental group $\Gamma_M$. 
We show that properties of $\Gamma_M$, and the isometric embedding 
$f_{g_M}: (M,g_{M}) \rightarrow (\mb{S}^{\tn},\tg)$ and its isotopic 
deformations, lead to the determination of the Kazdan-Warner type of $M$
and sigma invariant of its model induced differentiable structure. 
We compute $\sigma(M)$ of such $M$ in several specific cases with 
standard model geometries of Thurston.  Notably, we show that a manifold 
$M^n=\mb{H}^n/\Gamma_M$ of hyperbolic model $(\mb{H}^n,g_{\mb{H}^n})$ is of 
Kazdan-Warner 
type III, that if $n\geq 3$ its $\Gamma_M$ invariant  
hyperbolic metric $g_M$ and class realize its sigma invariant, 
and that the space of hyperbolic metrics on $M$ is path connected and 
consists of isotopic deformations $f_{g_t}$ of $f_{g_0}:=f_{g_M}$ of 
equal volume 
metrics $g_t$ of constant sectional curvature $-1$, with 
$(M,g_t)$ isometric to $(M,g_M)$ for all $t$, while $3$d nil, solv, 
$\widetilde{\mb{P}\mb{S}\mb{L}}(2,\mb{R})$ and $\mb{R}\times \mb{H}^2$ 
manifolds are all of Kazdan-Warner type III also, but have 
vanishing nonachievable sigma invariant. 
\end{abstract}
\author{Santiago R. Simanca}
\email{srsimanca@gmail.com}
\maketitle

\section{Introduction}
The group ${\rm Isom}(X,g)$ of isometries of a Riemannian manifold $(X^n,g)$ 
is a Lie group of dimension at most $n(n+1)/2$, with the upper bound 
achieved only if $g$ is of constant sectional curvature, and the 
isotropy subgroup ${\rm Isom}_p(X,g)$ of any $p\in X$ is closed \cite{mest}. 
If and when ${\rm Isom}(X,g)$ acts transitively on $X$, we 
may recast $(X,g)$ as the orbit space ${\rm Isom}\, (X,g)/{\rm Isom}_p (X,g)$ 
provided with its natural metric, exhibiting $(X,g)$ as a Riemannian
manifold that metrically looks the same about any point $p$ in it. We then 
say that $(X,g)$ is {\it homogeneous}. A complete Riemannian manifold $(M,g')$ 
is said to be {\it modeled} on a homogeneous manifold $(X,g)$ if every 
point of $M$ has a neighborhood that is isometric to an open set of the model 
$(X,g)$. We say then that $(M,g')$ as {\it locally homogeneous}. If
$(X,g)$ is the universal cover of $M$, and $g$ is the lift of $g'$, then 
$(X,g)$ is a homogeneous Riemannian manifold, $(M,g')$ is 
modeled on it, so the local geometry of $(M,g')$ is encoded in the 
geometry of $(X,g)$, and by completeness, the developing map is a covering 
map, and $M$ is of the form $X/\Gamma$ where $\Gamma$ is a discrete subgroup 
of ${\rm Isom}(X,g)$ acting without fixed points.  If ${\rm Isom}(X,g)$ 
is not contained in any larger group of diffeomorphism of $X$ with compact 
point stabilizers, we say that $({\rm Isom}(X,g),X)$ is a model geometry 
defining the geometric structure of the locally homogeneous manifold 
$(M,g')$. Conversely, a pair $({\rm Isom}(X,g),X)$ of this type is said 
to be a model geometry if there exists at least one compact manifold 
$M$ modeled on it, in which case, $M=X/\Gamma_M$ where $\Gamma_M$ is a 
discrete subgroup of ${\rm Isom}(X,g)$ acting on $X$ without fixed points, and 
the metric $g$ on $X$ passes to a metric $g_M$ on $M$ that is 
invariant under the action of $\Gamma_M \cong \pi_1(M)$.  
The model $(X,g)$ and $\Gamma_M$ define the differentiable structure 
of $M$, and together, yield a distinguished element $g_M$ in the cone 
$\mc{M}(M)$ of smooth metrics on it.

In Thurston's pioneering approach, compact manifolds of dimension $n$ are to be 
``classified'' by decomposing each of them into ``locally homogeneous pieces'' 
of finite volume modeled by suitably specified geometric models of the said 
dimension. If $n>3$, such an ambitious goal 
is difficult to achieve without additional restricting conditions, but it is 
straightforward if $n=1$, it follows from 
classical results if $n=2$, and Thurston's executes it for $n=3$, describing 
eight homogeneous manifolds admitting actions by discrete subgroups of quotient 
$3$d manifolds of finite volume, and proving the bulk of his geometrization 
theorem (TGT), originally \cite[Conjecture 1.1]{thurs0},   
that the interior of a compact 3d manifold have a canonical decomposition into 
pieces with geometric structures modeled by them (the proof completed 
after Perelman's solution \cite{per1,per2,per3}
of the subsumed elliptic case, originally asked for 
by Poincar\'e), the pieces here being the prime 
factors in the connected sum decomposition of the manifold, suitably cut 
when needed along 2-tori, with the resulting portions of 
finite volume modeled by one of the said eight geometries. One of our main 
purposes here is to show that if $M$ is a closed locally homogeneous 3d 
manifold modeled by anyone of the eight Thurston's 3d 
model geometries $({\rm Isom}(X^3,g),X^3)$,
so for a discrete subgroup $\Gamma_M$ of 
${\rm Isom}(X^3,g)$, we have that  
$M=X/\Gamma_M$, and this $M$ carries a $\Gamma_M$ invariant 
metric $g_M$, the data given by 
$\Gamma_M$ and $g_M$ suffice for the determination of the sigma invariant of 
$M$ and diffeomorphism type as a smooth $3$d manifold. In more generality, 
we use Thurston's vision restricted to the study of closed manifolds $M^n$ with
model geometries. Given $M$, we view the cone of smooth metrics $\mc{M}(M^n)$ 
and their deformations through the set 
of isometric embeddings of its elements into a standard sphere 
$(\mb{S}^{\tn=\tn(n)},\tg)$ and their Palais isotopic deformations, with the 
space $\mc{C}(M)$ of conformal classes corresponding to the set of classes of 
metrics whose embeddings can be deformed into each other in the said background 
by conformal isotopies. We characterize special metrics on these 
modeled $M$s through properties of their isometric embeddings, and 
in several relevant cases, show that the isometric embedding of $g_M$ itself,
or suitable isotopic deformations of it, provides $(M,\Gamma_M)$  
with an optimal geometry encoding a differentiable invariant
of the manifold, and often a canonical conformal class of metrics that 
realize it, and whose additional isometric invariants are not always, even in
dimension three, determined uniquely by the fundamental group.

If $f_g: (M^n,g) \rightarrow (\mb{S}^{\tn},\tg)$ is the isometric embedding
of a metric $g$ of scalar curvature $s_g$, we have that
$$
s_g = n(n-1) +\| H_{f_g}\|^2 -\| \alpha_{f_g}\|^2 \, ,  
$$
where $n(n-1)=\sum K^{\tg}(e_i,e_j)$, $H_{f_g}$ and $\alpha_{f_g}$ are the 
exterior scalar curvature, mean curvature vector and second fundamental form 
of $f_g$, respectively. (If $n\geq 2$, by the existence of volume preserving 
deformations of isometric embeddings of $g$, it follows that
$\| H_{f_g}\|^2$ is a constant function on $f_g(M)$ \cite[Theorem 6]{sim5}.)
These three summands are the densities of 
extrinsic functionals $\Theta_{f_g}(M)$, $\Psi_{f_g}(M)$ 
and $\Pi_{f_g}(M)$ such that, if
$\mc{W}_{f_g}(M) := (1-\delta_{n,1})(n/n-1)\Theta_{f_g}(M)+ \Phi_{f_g}(M)$ 
and $\mc{D}_{f_g}(M) := (1-\delta_{n,1})(1/n-1)\Theta_{f_g}(M)+ \Pi_{f_g}(M)$
(by abuse of notation we take $(1-\delta_{n,1})/(n-1)$ as $0$ when
$n=1$), we have that
$$ 
{\displaystyle \mc{S}_g(M)=\int s_g d\mu_g } = {\displaystyle  
\Theta_{f_g}(M)  
+ \Psi_{f_g}(M)  
- \Pi_{f_g}(M) } 
= \mc{W}_{f_g}(M) 
-\mc{D}_{f_g}(M) 
\, .  
$$ 
Both $\mc{W}_{f_g}(M)$ and $\mc{D}_{f_g}(M)$ are well-defined functionals 
in the space $\mc{M}_{[g]}(M)$ of metrics in a conformal class 
$[g] \in \mc{C}(M)$, unlike $\Theta_{f_g}(M)$, $\Psi_{f_g}(M)$ and 
$\Pi_{f_g}(M)$ that are so only in the special case when $M$ is a 
manifold that admits a single conformal class of metrics on it.
The sigma invariant of $M$ is associated to 
conformal classes for which there exists simultaneous minimizers $f_g$ of 
$\mc{W}_{f_g}(M)$ and $\mc{D}_{f_g}(M)$ with $f_g$ minimal, and $g$ 
a metric of constant scalar curvature. If and when it is realized, 
$\sigma(M)$ is determined by a geometrically optimal such class. 

If $n=1$, $M^n$  is topologically a circle, it carries one 
conformal class of metrics $[g]$ all of whose representatives $e^{2u}g$ are 
flat, and an embedding $f_g$ 
minimizes $\mc{W}_{f_g}(M)$ and is minimal if, and only if, 
$f_g(M) \subset \mb{S}^{\tn}$ is a geodesic circle, 
and $\mu_g(M)=2\pi$. We then define $\mc{W}(M,[g])=2\pi = \mc{D}(M,[g])$, and 
$\sigma(M)=0=\mc{W}(M,[g]) - \mc{D}(M,[g])$.
For $n\geq 2$, we instead define $\mc{W}(M,[g])=\inf_{g\in \mc{M}_{[g]}(M)}
\mc{W}_{f_g}(M)$ and $\mc{D}(M,[g])=\inf_{g\in \mc{M}_{[g]}(M)}
\mc{D}_{f_g}(M)$, respectively. 
 If $n=2$, $\mc{W}_{f_g}(M)$ and $\mc{D}_{f_g}(M)$ are 
conformally invariant functionals, $\mc{W}_{f_g}(M)=\mc{W}(M,[g])$ 
and $\mc{D}_{f_g}(M)=\mc{D}(M,[g])$,
and by the Gauss-Bonnet theorem, we have that
$4\pi \chi(M)=\mc{W}_{f_g}(M)-\mc{D}_{f_g}(M)$. 
In each conformal class $[g]$ there is always a metric $g$ with 
minimal $f_g$ such that $\mc{W}_{f_g}(M)=\mc{W}(M,[g])$, and if necessary, 
this $f_g$ can be conformally deformed by area preserving minimal embeddings 
to the isometric embedding of an Einstein metric \cite[Theorem 1]{sim8}. 
A surface $M$ of topological genus $k$ carries a distinguished 
conformal class $[g_k]\in \mc{C}(M)$ such that
$\mc{W}_{f_g}(M)=\mc{W}(M,[g]) \geq \mc{W}(M,[g_k])$ 
\cite[Theorems 1 \& 9]{sim2}, and the invariant  
$\sigma(M):= 8\pi \chi(M)/\mc{W}(M,[g_k])^{\frac{1}{2}}$   
can be taken to be realized by an optimal minimal $f_{g_k}$ of a metric $g_k$
of constant scalar curvature, even when $\chi(M)=0$, in which case any other 
metric $g$, minimizer or not of $[g]\rightarrow \mc{W}_{f_g}(M)=\mc{W}(M,[g])$, 
realizes the said value as well. 
If $n\geq 3$, the Yamabe functional is defined by the scale normalized total
scalar curvature 
\begin{equation}  \label{yafu} 
\lambda(M,g)= \frac{1}{\mu_g(M)^{\frac{n-2}{n}}} \left( \mc{S}_g(M) = 
\Theta_{f_g}(M)  + \Psi_{f_g}(M) - 
\Pi_{f_g}(M) 
=\mc{W}_{f_g}(M) -\mc{D}_{f_g}(M)\right) \, ,  
\end{equation} 
and its critical points over metrics of fixed volume in a conformal class
are metrics of constant scalar curvature. The minimizers in a class $[g]$ 
are the Yamabe metrics in $\mc{M}_{[g]}(M)$, and the critical value     
$$
\lambda(M^n,[g])= \inf_{g\in \mc{M}_{[g]}(M)} \lambda(M,g) 
$$
is a conformal invariant of $[g]\in \mc{C}(M)$
that is always realized by a $g\in \mc{M}_{[g]}(M)$, hence by any 
scale deformation $c^2g$ of it. The optimal scale of a Yamabe metric $g$ may 
be found by observing that $g$ is then either Einstein or scalar flat,  
in which case, if necessary, it can be dilated by the factor
\begin{equation} \label{minf}
c^2_{min}=c^2_{min}(g)=\left( 1 + \frac{1}{n^2}\| H_{f_g} \|^2 \right)
\end{equation}
so that the embedding $f_{c_{min}^2 g}: (M,c_{min}^2 g) \rightarrow
(\mb{S}^{\tn},\tg)$ is minimal, or otherwise, $s_g$ is constant, and $f_g$ is
outright a minimal embedding that  
minimizes both, $\mc{W}_{f_g}(M)$ and $\mc{D}_{f_g}(M)$.  
\cite[\S 2]{sim7}, \cite[Theorem 3]{sim6}. 
In any of these three possible situations, if $g$ is the Yamabe metric in 
$\mc{M}_{[g]}(M)$ of minimal $f_{g}$, which is unique up to volume preserving 
conformal isometries and therefore, fixes the canonical scale of Yamabe 
metrics in the class, 
we have that 
$$
\lambda(M,[g])=s_{g}\mu_{g}(M)^{\frac{2}{n}} =
\frac{1}{\mu_g(M)^{\frac{n-2}{n}}}\left( \mc{W}_{f_g}(M)=\mc{W}(M,[g])-
\mc{D}_{f_g}(M)=\mc{D}(M,[g]\right) \, . 
$$ 
We have Aubin's universal bound \cite{au}
\begin{equation} \label{aub}
\lambda(M^n,[g]) 
\leq \lambda(\mb{S}^n, \tg)=n(n-1)\omega_n^{\frac{2}{n}} =:\sigma(\mb{S}^n) \, ,
\end{equation}
and as the concept was originally introduced by Schoen \cite{sc2}, 
the sigma invariant of $M$ is given by 
$$
\sigma(M^n)= \sup_{[g]\in \mc{C}(M)}\lambda(M,[g]) \leq
\sigma(\mb{S}^n)\, .  
$$
Its realizability now is a more complicated issue, since knowledge of a 
particular Yamabe metric on $M$ is far from being enough information to 
determine even the Kazdan-Warner type of the manifold, let alone its 
$\sigma$ invariant. However, the bundle $\mc{M}(M) \rightarrow \mc{C}(M)$ has a 
homotopy lifting property by Yamabe metrics \cite[Theorem 3]{sim6}, and we 
can exploit
the special nature of the isometric embeddings of paths of these metrics 
to show that, under suitable geometric conditions on $M$, deformations of
metrics of this type across conformal classes is sufficiently obstructed in
a way that allows us to describe firstly the Kazdan-Warner type of the
manifold, mimicking the situation for surfaces and circle, and then attempt 
to compute the sigma invariant of some manifolds with suitable model 
geometries on the basis of the insight so gained. 
We cast the former type of results we obtain into statements that are 
applicable in general, subsumming in them the properties we already know of 
in dimensions one and two.

If $g'$ is a metric on $M$ of nontrivial scalar curvature $s_{g'}\geq 0$, then
there exists a metric $g\in [g']$ of constant positive scalar curvature, 
$M$ is a manifold of Kazdan-Warner type I, and $\sigma(M)>0$. Suppose now 
that $g$ is a metric on $M^n$ of constant scalar 
curvature $s_g$ and Ricci tensor $r_g\leq 0$, and that if $n\geq 3$, $M$ is 
a manifold that carries no scalar flat metrics other than Ricci flat ones, 
if any. We then prove that $M$ cannot be of Kazdan-Warner type I, and if the
space of Ricci flat metrics on it is not empty, $M$ is of Kazdan-Warner type 
II, while otherwise, it is of Kazdan-Warner type III.
It follows that  if $M$ is a manifold with contractible universal cover that 
carries no Ricci flat metric at all, $M$ must be of Kazdan-Warner type III. 
We can then show that if a manifold $M^n$ 
carries a metric $g'$ of nontrivial $s_{g'}\geq 0$, and an Einstein metric 
$g_{-}$ of $r_{g_{-}}<0$, then $M$ must carry both, scalar flat non Ricci flat 
and Ricci flat metrics, and if orientable, it is spinnable, this last result 
following by using a Ricci flat metric to compute a metric representative of 
$w_2(M)$ and show that it vanishes in cohomology.  No such manifold exists if 
$n\leq 3$, and if $g'$ is assumed further to be such that $r_{g'}>0$, no such 
manifold exists 
if $n\leq 4$, and an $M^n$ of these dimensions can admit Einstein metrics of 
scalar curvature of at most one sign. For then the existence of Einstein 
metrics of 
different signs forces the condition $n>3$, as otherwise, $M$ would be 
diffeomorphic to closed space forms of different curvatures. If $n=4$, the 
condition $r_{g'}>0$,  
makes it possible to apply Meyer's theorem, and provide the orientable universal
cover with a metric of positive scalar curvature, which by
Lichnerowicz theorem would have to be of zero signature, and an Einstein 
metric of negative Ricci curvature. By Freedman's classification theorem in the 
topological category, this cover would then have to be either 
$\mb{S}^4$, or a connected sum $\#_{i=1}^k \mb{S}^2 \times \mb{S}^2$, with 
differentiable structures compatible with that of $M$, and carrying 
a Ricci flat metric that we then use conveniently to define an almost complex 
structure on it, thus excluding from consideration all but $\mb{S}^2 \times 
\mb{S}^2$, as the other possibilities are manifolds obstructed to carry 
almost complex structures by the integrality of the arithmetic genus. The 
remaining case of 
$\mb{S}^2\times \mb{S}^2$ has a unique differentiable structure, and 
on it, we can explicitly produce an Einstein metric of positive scalar 
curvature, but this differentiable structure is not compatible with 
the differentiable structure of the homeomorphic manifold that
carries the Ricci flat metric on it,
since it would be then impossible to produce a path of Yamabe 
metrics connecting the Einstein metric of positive scalar curvature with it, 
contradicting the homotopy lifting property by Yamabe metrics of the
bundle $\mc{M}(M) \rightarrow \mc{C}(M)$. On the basis of this result, 
we show that if $M^{n\leq 4}$ carries Einstein metrics at all, the sign of 
the scalar curvatures of the set of these metric on $M$ is unique. 
On the other hand,  smooth paths of 
differentiable isometric embeddings $[0,1]\ni t \rightarrow f_{g_t}: (M^n,g_t) 
\rightarrow ( \mb{S}^{\tn},\tg)$ of smooth paths $t\rightarrow g_t$ of 
differentiable Yamabe metrics of isometry group of maximal dimension each  
are quite special. For we have then that
$t\rightarrow g_t \rightarrow s_{g_t}$ is differentiable, and then 
that the  sectional curvature of $(f_{g_{t}}(M), \tg\mid_{f_{g_{t}}(M)})$ 
is constant for some $t$ if, and only if, it is constant for all $t$, in which 
case, the sign of the constant functions in the differentiable path of 
sectional curvatures is unique. 
This generalizes what happens in dimensions one and two, 
with the proper extension of the concepts we have introduced earlier. For 
smooth deformations of a geodesic in $(\mb{S}^{\tn},\tg)$ are necessarily 
of intrinsic vanishing scalar curvatures, while smooth deformations across 
conformal classes of metrics of constant scalar curvature on a surface $M$ 
maintain the sign of the constant,
which is determined by that of $\chi(M)$, and where the differentiability of 
the path of scalar curvatures follows as an application of the  
extreme case of the Cauchy Schwarz inequality. 

A general principle for the computation of $\sigma(M)$ 
begins by the observation that scalar flat metrics on an $M$ of Kazdan-Warner 
type II are Ricci flat, any of them realizes the invariant, and that there is 
one such $g$ that is optimal in that $f_{g}$ is minimal, and 
$\mc{W}_{f_g}(M) \leq \mc{W}_{f_{g'}}(M) = 
\mc{D}_{f_{g'}}(M) \geq  \mc{D}_{f_{g}}(M)$ among all the Ricci flat metrics 
$g'$ on $M$, thus extending to all $n$s what happens with this type of manifolds
in 1 and 2d \cite[Theorem 1]{sim7}. When it can be found explicitly, 
 this $g$ defines a 
canonical volume $\mu_g(M)$ and smallest value of $\| \alpha_{f_g}\|^2$ 
among all Ricci 
flat metrics on the manifold, instances of which we have illustrated 
already for each 
of the ten manifolds of Euclidean 3d model geometry \cite[Theorem 5]{sim7}.  
The situation for manifolds of Kazdan-Warner
type I and III is markedly different, and even among those whose invariant is
achieved, it is not always the case that the achievable conformal class
carries Einstein metrics in it \cite[Corollary 3]{sim4}, and an       
asymmetric distinction between manifolds of type I and III is reflected
in the differentiability of the scalar curvature function on the lift by 
constant scalar curvature metrics of a smooth path of conformal classes, 
where if the metrics are of nonpositive scalar curvature, the scalar 
curvature varies differentiably across classes, in contrast with the 
positive alternative where if $n\geq 3$ and the lift is by Yamabe metrics, 
only its continuity is always ensured \cite[Theorem 3]{sim6}. 
In the type I case, we may use Aubin's bound (\ref{aub}) and a continuity
based method to compare equal volume Yamabe metrics of positive
scalar curvature, by which we can often identify a Yamabe metric $g$ of 
sufficiently large volume and minimal $f_g$ so that no other 
equal volume Yamabe metric has scalar curvature larger than $s_g$, and so
$\lambda(M,[g])=\sigma(M)$, e.g. \cite[Theorems 2 \& 5]{sim4}. We have applied 
this method, and the explicit description of all finite subgroups of 
$\mb{S}\mb{O}(4)$, to find the value of the invariant, realizing class, 
and canonical volume of its representative with minimal isometric embedding
for all elliptic 3d manifolds as well, the canonical volume then an explicit 
function of $|\pi_1(M)|$ \cite[Theorem 6]{sim7}. We extend this result here 
to Lens spaces of any dimension. On a manifold $M^{n\geq 3}$ of hyperbolic 
model $(H^n,g_{H^n})$, 
we use the Mostow rigidity theorem, and the rigidity of differentiable paths
of homotopy lifts by Yamabe metrics of constant volume metrics that start at 
a metric of constant sectional curvature, 
to show that the hyperbolic metric $g_M$ on $M=H^n/\Gamma_M$ and 
its conformal class realize $\sigma(M)=-n(n-1) \mu_{g_M}(M)^{\frac{2}{n}}$, 
that on any class $[g] \in \mc{C}(M)$ there are metrics of negative constant 
sectional curvature representatives, and that the space of hyperbolic metrics 
on this $M$ 
is path connected and consists of isotopic deformations of $g_M$ by 
equal volume hyperbolic metrics that are all isometric to each other, but since 
the structure of the associated subgroup $\Gamma_M \subset {\rm Isom}(H^n,
g_{H^n})$ is not now as explicit as in the previous cases, we cannot determine 
the volume of a canonical metric of constant sectional curvature with 
minimal isometric embedding, hence, we are unable to 
determine explicitly the canonical volume of $M$. 
By the uniqueness of the differentiable structure of a manifold of dimension 
$n\leq 3$, when $n=3$, it follows that ${\rm Diff}_0(M)$ is contractible, a 
particular case, with a different proof,  
of D. Gabai's \cite[Theorem 7.3]{gab} for a hyperbolic 3d 
manifold of finite volume, and which does not extend to $n\geq 4$, as
for instance, if $n=4$, then the set $\pi_0({\rm Diff}_0(M))$ is infinitely 
generated and there are 
hyperbolic metrics in different path components of ${\rm Diff}_0(M)$ 
\cite{buga} \footnotemark. 
We further discuss the invariant $\sigma(M)$ for 3d manifolds with any 
of the remaining five model geometries of Thurston, and in some 
cases, their $n$ dimensional versions as well.   
\footnotetext{We thank D. Gabai for pointing out references \cite{gab,buga} to 
us in response to a question of ours.}
 
We organize the article as follows: In \S2, we prove our technical results 
of earlier, characterizing the Kazdan-Warner type of a manifold under 
suitable geometric assumptions on it, and proving that any $M^{n\leq 4}$ can 
carry Einstein metrics of at most one sign, if at all. 
We show here also that smooth paths of differentiable isometric embeddings
of smooth paths of differentiable Yamabe metrics can be of constant sectional 
curvature at some value of the path parameter if, and only if, they are so 
for all such. 
In \S3, we first describe the standard geometric models $({\rm Isom}(X^n,g),
X^n)$ of Thurston in full dimensional generality, and then discuss the $\sigma$ 
invariant of a few manifolds with these model geometries, in particular, the 3d 
manifolds with these geometric models in Thurston's TGT.
We show that the results of \S2 immediately lead to the 
conclusion that the hyperbolic metric and class of a hyperbolic $n$ manifold 
$M$ of fundamental group $\Gamma_M$ realize its sigma invariant, and that 
the space of hyperbolic metrics on $M$ is path connected. We also prove
then that Heisenberg group nilmanifolds are of 
Kazdan-Warner type III with nonrealizable vanishing $\sigma$ invariant, 
properties that are shared by 3d solv, $\widetilde{\mb{P}\mb{S}\mb{L}}(2,
\mb{R})$ and $\mb{R}\times \mb{H}^2$ manifolds also. 
We end in \S4 by computing the 
$\sigma^{J_0}(M)$ invariant of several 4d almost complex manifolds $(M,J_0)$, 
comparing it in each case with the invariant $\sigma(M)$. 
This illustrate some subtleties arising  
in the classification of manifolds as we transition in dimension
to $n=2m \geq 4$, where the appearance of nontrivial $(M,J_0)$s  
play a distinctive role, even when the a priori fixed $M$ has
a geometric model that is the product of lower dimensional ones, and 
carries Einstein metrics.

\section{Constant scalar curvature metrics and Kazdan-Warner type}
\begin{theorem} \label{thty}
Suppose that $g$ is a metric on $M^n$ of constant scalar curvature $s_g$ and
Ricci tensor $r_g \leq 0$. If $n\geq 3$, assume further 
that any scalar flat metric on
$M$ is Ricci flat. Then $M$ is not a manifold of Kazdan-Warner type 
{\rm I}, and if the space of Ricci flat metrics on $M$ is not empty, $M$ is a 
manifold of Kazdan-Warner type {\rm II}, while otherwise, $M$ is a manifold of 
Kazdan-Warner type {\rm III}.   
\end{theorem}

{\it Proof}. If $n<3$, the statement holds by the definition of 
Kazdan-Warner type for manifolds of such dimensions \cite[\S 2]{sim6}. 
So we proceed assuming that $n\geq 3$. 
Any metric of nonpositive constant scalar curvature is
a Yamabe metric in its conformal class. We prove firstly that  
the hypothesis on the scalar flat metrics of $M^{n}$ imply that $M$ does not 
carry metrics of positive scalar curvature, and therefore, it is not 
a manifold of Kazdan-Warner type I. 
 
For if $M^{n}$ carries a metric of nonnegative scalar curvature 
that is positive somewhere, then $M$ carries Yamabe metrics of constant 
positive scalar curvature. We take $g'$ to be one such metric, and consider
a path $[0,1]\ni t \rightarrow g_t$ of metrics connecting
$g$ to $g'$, and lift of the path of classes 
$t \rightarrow [g_t]$ to an associated path $t\rightarrow g^Y_t$ of 
$\mu_{g_t}$ volume Yamabe metrics in $[g_t]$
such that $g^Y_0=g$ and $g^Y_1=g'$ \cite[Theorem 3]{sim6}. By the continuity 
of $t \rightarrow \lambda(M,[g^Y_t])=\lambda(M,g^Y_t)$ and the intermediate 
value theorem, there exists a first $\bar{t} \in [0,1)$ and an $\varepsilon >0$
such that $[0,\bar{t}+\varepsilon)\subset [0,1]$, 
$(0,\bar{t})\ni t \rightarrow s_{g_t^Y} \leq 0$,  $s_{g_{\bar{t}}^Y}=0$, 
and $(\bar{t},\bar{t}+\varepsilon) \ni t \rightarrow s_{g_t^Y}> 0$
and $(\bar{t},\bar{t}+\varepsilon) \ni t \rightarrow s_{g_t^Y}/(n-1) \not
\in {\rm Spec}\, \Delta^{g^Y_t}$. 

By the assumed hypothesis on $M$, the
metric $g^Y_{\bar{t}}$ is Ricci flat.   
If necessary, we homothetically deform 
the Yamabe metric $g$ by  a factor $c^2=c^2_{min}(g)$ in (\ref{minf}) so that 
the isometric embedding 
$f_{c^2 g}: (M,g) \rightarrow (\mb{S}^{\tn}, \tg)$ is minimal, and consider 
the rescaled 
path $t\rightarrow \tilde{g}_t^Y:=(\mu_{c^2g}/\mu_{g_t^Y})^{\frac{2}{n}} g_t^Y$ 
of constant $\mu_{c^2g}$ volume Yamabe metrics in $[g_t^Y]$ 
connecting $\tilde{g}^Y_0=c^2 g$ with $\tilde{g}^Y_1=
(\mu_{c^2g}/\mu_{g_t^Y})^{\frac{2}{n}}g'$,   
path of isometric embeddings $t\rightarrow f_{\tilde{g}_t^Y}$, and
corresponding paths of functionals $t\rightarrow \lambda(M^n,f_{\tilde{
g}_t^Y})$, $t\rightarrow \mc{W}_{f_{\tilde{g}_t^Y}}(M^n)$, 
$t\rightarrow \mc{D}_{f_{\tilde{g}_t^Y}}(M^n)$,
$t\rightarrow \Phi_{f_{\tilde{g}_t^Y}}(M^n)$, and  
$t\rightarrow \Pi_{f_{\tilde{g}_t^Y}}(M^n)$,
respectively. By construction we have that
$(0,\bar{t})\ni t \rightarrow s_{\tilde{g}_t^Y} \leq 0$, $s_{\tilde{g}_{
\bar{t}}^Y}=0$, $\tilde{g}^Y_{\bar{t}}$ is Ricci flat, and   
$(\bar{t},\bar{t}+\varepsilon) \ni t \rightarrow s_{\tilde{g}_t^Y}> 0$
and $(\bar{t},\bar{t}+\varepsilon) \ni t \rightarrow s_{\tilde{g}_t^Y}/(n-1) 
\not \in {\rm Spec}\, \Delta^{\tilde{g}^Y_t}$ so $(0,\bar{t}+\varepsilon)
\ni t \rightarrow s_{g_t^Y}$ is differentiable. 

Over the interval $(0,\bar{t}+\varepsilon)$, the functionals 
$t\rightarrow \lambda(M,\tilde{g}_t^Y)$, $t\rightarrow \Psi_{f_{\tilde{g}_t^Y}
}(M)$ and $t\rightarrow \Pi_{f_{\tilde{g}_t^Y}}(M)$ are differentiable, and 
since the variations are by constant volume metrics, by (\ref{yafu}) 
we obtain that 
\begin{equation} \label{veq}
\begin{array}{rcl}
{\displaystyle \frac{d}{dt} \lambda(M,\tilde{g}^Y_t) =- 
\frac{1}{\mu_{c^2g}^{\frac{2}{N}}} \int (r_{\tilde{g}_t^Y},h_t)_{\tilde{g}_t^Y}
d\mu_{\tilde{g}_t^Y}} & = & {\displaystyle \frac{1}{\mu_{c^2g}^{\frac{2}{N}}}
\left( \frac{d}{dt} \Psi_{f_{\tilde{g}^Y_t}}(M) - \frac{d}{dt} 
\Pi_{f_{\tilde{g}^Y_t}}(M) \right) } \vspace{1mm} \\ & = & {\displaystyle  
\frac{1}{\mu_{c^2g}^{\frac{2}{N}}}\frac{d}{dt}\left( 
\mc{W}_{f_{\tilde{g}_t^Y}}(M) - \mc{D}_{f_{\tilde{g}_t^Y}}(M)\right) } \, , 
\end{array}
\end{equation}
where $N=2n/(n-2)$, $r_{\tilde{g}_t^Y}$ is the Ricci tensor of 
$\tilde{g}_t^Y$, and  $h_t=\dot{\tilde{g}}^Y_t$ is a symmetric tensor of zero 
$\tilde{g}^Y_t$-trace.
As $\tilde{g}_t^Y$ moves across conformal classes, this 
expression implies that $s_{\tilde{g}_t^Y}$ can become positive to the 
immediate right of $\bar{t}$ through Yamabe metrics such that $(\bar{t}, 
\bar{t}+\varepsilon)\ni t \rightarrow f_{\tilde{g}_t^Y}$ is a minimal minimizer 
of both, $\mc{W}_{f_g}$ and $\mc{D}_{f_g}$. Otherwise, $(\bar{t},
\bar{t}+\varepsilon) \ni 
t \rightarrow \tilde{g}_t^Y$ would have to be a path of Einstein metrics of 
positive scalar curvature, and by continuity, a path of such metrics on 
$[\bar{t},\bar{t}+\varepsilon)$, but then the left side of (\ref{veq})  
would imply that 
$\lambda(M,\tilde{g}_t^Y)$, which has value $0$ at $t=\bar{t}$, has vanishing
differential, and so vanishes identically on 
$[\bar{t},\bar{t}+\varepsilon)$, as must then do $s_{\tilde{g}_t^Y}$, 
contradicting its positivity. 
By using now the right side of (\ref{veq}) instead, we show that  
this minimality leads to the same contradiction.

Indeed, since the variation is by Yamabe metrics, the functionals 
$\Psi_{f_{\tilde{g}_t^Y}}(M)$ and $\Pi_{f_{ \tilde{g}_t^Y}}(M)$ are each 
stationary in the directions of the normal bundle $\nu(f_{\tilde{g}_t^Y}(M))$ 
of the submanifold $f_{\tilde{g}_t^Y}(M)$ 
inside the background sphere $\mb{S}^{\tn}$, so if we let $T=df_{\tilde{g}_t^Y}
(\partial_t)=T^{\tau}+T^{\nu}$ be the decomposition of the variational vector 
field of the embedding $f_{\tilde{g}_t^Y}$ into tangential and normal 
components, their variations vanish in the direction of $T^{\nu}$, 
and so they equal their variations along $T^{\tau}$. By \cite[Theorems 3.1 \& 
3.2]{gracie}, these variations are thus given by
\begin{equation} \label{grvea}
\begin{array}{rcl}
{\displaystyle \frac{d \Pi_{f_{\tilde{g}_t^Y}}(M)}{dt} } & = & {\displaystyle
\int_{f_{\tilde{g}_t^Y}(M)} 2 \< \nabla^{\tg}_{e_j}\nabla_{e_i}^{\tg}
\alpha_{f_{\tilde{g}_t^Y}}(e_i,e_j),T^{\tau}\> d\mu_{\tilde{g}_t^Y} } \, , 
\vspace{1mm}\\
& & -{\displaystyle \int_{f_{\tilde{g}_t^Y}(M)} 
2(e_i\< T^{\tau},e_l\> +e_l\<T^{\tau},e_i\>) \< \alpha_{f_{\tilde{g}_t^Y}}
(e_i, e_j), \alpha_{f_{\tilde{g}_t^Y}}(e_l,e_j)\> d\mu_{\tilde{g}_t^Y} } \, , 
\end{array}
\end{equation}
and 
\begin{equation} \label{grvep}
\begin{array}{rcl}
{\displaystyle \frac{d \Psi_{f_{\tilde{g}_t^Y}}(M)}{dt} } & = & {\displaystyle
\int_{f_{\tilde{g}_t^Y}(M)} 2 \< \nabla_{e_i}^{\tg}\nabla_{e_i}^{\tg}H_{f_{
\tilde{g}g_t^Y}}, T^{\tau}\> d\mu_{\tilde{g}_t^Y} } \, , \vspace{1mm} \\
& & -2{\displaystyle \int_{f_{\tilde{g}_t^Y}(M)} (e_i\< T^{\tau},e_j\> +
e_j\<T^{\tau}, e_i\>) \< \alpha_{f_{\tilde{g}_t^Y}} (e_i, e_j),H_{f_{
\tilde{g}_t^Y}} \>d\mu_{\tilde{g}_t^Y} }\, , 
\end{array}
\end{equation}
respectively, expressions in which $\{ e_i=e_i^t\}$ is an orthonormal 
$\tilde{g}_t^Y$ frame in a sufficiently small neighborhood of the integral 
curve of $T$ through the point $(p,t)$ where the densities of the integrals are 
computed, at which in addition we have that $[T,e_i]\mid_{(p,t)}=0=
\nabla^{g_t^Y}_{e_i}e_j\mid_{(p,t)}$. 
 
By (\ref{grvep}) and the minimality of $(\bar{t},\bar{t}+\varepsilon)\ni t 
\rightarrow f_{\tilde{g}_t^Y}$, we obtain that
\begin{equation} \label{veqh}
(\bar{t},\bar{t}+\varepsilon) \ni t \rightarrow 
{\displaystyle \frac{d \Psi_{f_{\tilde{g}_t^Y}}(M)}{dt} } = 0 \, .
\end{equation}
On the other hand, 
by Codazzi's equation, we have that  
$$
\nabla_{e_i}^{\tg}\alpha_{f_{\tilde{g}_t^Y}}(e_i,e_j)=(R^{\tg}(e_i,e_j)e_i)^{
\nu}+ \nabla^{\tg}_{e_j} H_{f_{\tilde{g}_t^Y}}+ A_{H_{f_{\tilde{g}_t^Y}}}e_j-
A_{\alpha_{f_{\tilde{g}_t^Y}} (e_j,e_i)}e_i \, , 
$$
where $A$ is the shape operator of the embedding, and we obtain that
$$
\begin{array}{rcl}
{\displaystyle \int_{f_{\tilde{g}_t^Y}(M)} \< \nabla^{\tg}_{e_j}\nabla_{e_i}^{
\tg} \alpha_{f_{\tilde{g}_t^Y}}(e_i,e_j),T^{\tau}\> d\mu_{\tilde{g}_t^Y} } & = 
& {\displaystyle  \int_{f_{\tilde{g}_t^Y}(M)} \left( \< H_{f_{\tilde{g}_t^Y}}, 
\nabla^{\tg}_{e_j}\nabla_{e_i}^{\tg} T^{\tau}\>
- \< A_{H_{f_{\tilde{g}_t^Y}}}e_j, \alpha_{f_{\tilde{g}_t^Y}}(e_i,e_j),
T^{\tau}\>\right) d\mu_{\tilde{g}_t^Y}} \\   &  & {\displaystyle +\int_{f_{
\tilde{g}_t^Y}(M)} \< \alpha_{f_{\tilde{g}_t^Y}}(e_i,e_j), \alpha_{f_{
\tilde{g}_t^Y}}(e_i,e_l)\> e_j\< e_l,T^{\tau}\> d\mu_{\tilde{g}_t^Y} } \, .    
\end{array}
$$
Thus, by (\ref{grvea}) and the minimality of $f_{\tilde{g}_t^Y}$, we 
conclude that 
\begin{equation} \label{veqp}
(\bar{t},\bar{t}+\varepsilon) \ni t \rightarrow
{\displaystyle \frac{d \Pi_{f_{\tilde{g}_t^Y}}(M)}{dt} 
= -2 \int_{f_{\tilde{g}_t^Y}(M)} \< \alpha_{f_{\tilde{g}_t^Y}}(e_i,e_j),
\alpha_{f_{\tilde{g}_t^Y}}(e_i,e_l)\> e_l\< e_j,T^{\tau}\> d\mu_{\tilde{g}_t^Y}
} = 0\, , 
\end{equation}
the last vanishing equality because at the point where the density of the 
integral 
is computed, the tangential frame $\{ e_i\}$ is taken to be 
geodesic and all the $e_i$s 
commute with $T$. 
By (\ref{veq}), (\ref{veqh}) and (\ref{veqp}), $\lambda(M,\tilde{g}_t^Y)$ is 
stationary over the interval $(\bar{t},\bar{t}+\varepsilon)$, and therefore, 
identically zero on $[\bar{t},\bar{t}+\varepsilon)$, contradicting the fact 
that $s_{\tilde{g}_t^Y}>0$ for $t \in (\bar{t},\bar{t}+\varepsilon)$. Hence,  
$s_{\tilde{g}_t^Y}$ can never become positive on the interval $[0,1]$, which 
contradicts the fact that $s_{\tilde{g}_1^Y}$ is positive, and this conclusion
for $\tilde{g}_t^Y$ implies the same for the unnormalized $g_t^Y$. This proves 
that $M$ does not carry metrics of nonnegative scalar curvature other than 
possibly Ricci flat metrics, and so $M$ is not a manifold of Kazdan-Warner
type I.

If the space of Ricci flat metrics on $M$ is not empty, any such metric
would have zero scalar curvature and conformal class with zero Yamabe 
invariant. Hence, $M$, which is not of Kazdan-Warner type I, must be then of 
Kazdan-Warner type II. If otherwise the said space is empty, the manifold
does not carry zero scalar curvature metrics, and $M$, which is 
not of Kazdan-Warner type I, is not of Kazdan-Warner type II either. So 
$M$ is then a manifold of Kazdan-Warner type III. 
\qed

\begin{corollary} \label{co3}
If $M^n$ is a closed manifold with contractible universal cover
that carries no Ricci flat metric at all, then $M$ does not carry scalar
flat metrics either, and it is an $n$d manifold of Kazdan-Warner type 
{\rm III}.
\end{corollary}

{\it Proof}. The hypothesis implies that $n>1$. By \cite[Theorem 1]{sim7}, 
$M$ cannot be of Kazdan-Warner type II. 
If $n=2$, $M$ must be of type III, as the type is then determined by
the sign of $\chi(M)$, and the surfaces for which $\chi(M)>0$ do not have
contractible universal covers. If $n\geq 3$, there exists a metric 
$g$ of constant scalar curvature and negative Ricci tensor $r_g <0$
\cite{lohk}, so
$g$ is a Yamabe metric in its conformal class.  If $g'$ were
a Yamabe metric of positive scalar curvature, and $t \rightarrow g_t^Y$ 
a path of Yamabe metrics connecting $g$ to $g'$, we let $\bar{t}$ the  
last time this path reaches a metric of zero scalar curvature, rescale 
$g_{\bar{t}}^Y$ by the factor $c^2=c^2_{min}(g^Y_{\bar{t}})$ in (\ref{minf}) 
so that $f_{c^2 g^Y_{\bar{t}}}$ is minimal, and consider the path
$[0, 1]\ni t \rightarrow \tilde{g}_t^Y:=(\mu_{c^2g_{\bar{t}^Y}}/\mu_{g_t^Y})^{
\frac{2}{n}} g_t^Y$ of constant valume Yamabe metrics instead. By 
construction, 
for some $\varepsilon >0$, the path restricted to
$[\bar{t},\bar{t}+\varepsilon]$ consists of constant volume non Ricci flat
metrics of minimal embedding $f_{\tilde{g}_t^Y}$ each, 
$s_{\tilde{g}_{ \bar{t}}^Y}=0$, 
$(\bar{t},\bar{t}+\varepsilon) \ni t \rightarrow s_{\tilde{g}_t^Y}> 0$
and $(\bar{t},\bar{t}+\varepsilon) \ni t \rightarrow s_{\tilde{g}_t^Y}/(n-1) 
\not \in {\rm Spec}\, \Delta^{\tilde{g}^Y_t}$ so $(\bar{t},\bar{t}+\varepsilon)
\ni t \rightarrow s_{g_t^Y}$ is differentiable. By using the right side
of (\ref{veq}) as above, we then prove that $\lambda(M,\tilde{g}_t^Y)$, which 
vanishes at $t=\bar{t}$, is stationary on $(\bar{t},\bar{t}+\varepsilon)$,
so identically zero on $[\bar{t},\bar{t}+\varepsilon]$,     
 contradicting its positive value at $\bar{t}+\varepsilon$.
Thus, there are no Yamabe metrics $g'$ on $M$ with $s_{g'}>0$, and 
$M$ is neither of Kazdan-Warner type I nor of Kazdan-Warner type II.
\qed

\begin{theorem} \label{th2}
If $M^n$ carries a metric $g'$ of nontrivial scalar curvature $s_{g'}\geq 0$,
and an Einstein metric $g_{-}$ of Ricci tensor $r_{g_{-}}<0$, then $M$ 
carries both, scalar flat non Ricci flat and Ricci flat metrics, and if 
orientable, it is spinnable. No such manifold exists if $n\leq 3$, and if 
$g'$ is assumed to be such that $r_{g'}>0$, no such manifold exists if 
$n\leq 4$, and
in these dimensions, $M$ can admit Einstein metrics of scalar curvatures  
of at most one sign.
\end{theorem}

{\it Proof}. If $g'$ is not a Yamabe metric, we let $g_{+}$ be a Yamabe 
metric in the conformal class $[g']$, while otherwise, we set 
$g_{+}=g'$. so we have that $s_{g_{+}}$ is a positive constant. The existence of
$g_{+}$ and $g_{-}$ already implies that $n>2$. We let $[0,1]\ni t 
\rightarrow g_t$ be any path of metrics connecting $g_{-}$ to $g_{+}$, and 
consider the lift of the path of conformal classes $t \rightarrow [g_t]$ to a 
path $t\rightarrow g_t^Y$ of Yamabe metrics in $[g_t]$ such that $\mu_{g_t}=
\mu_{g^Y_t}$ connecting $g_{-}$ and $g_{+}$. By the continuity of 
$t \rightarrow \lambda(M,[g^Y_t])=\lambda(M,g^Y_t)$ and the intermediate 
value theorem, there exists a first $\bar{t} \in (0,1)$ such that 
$\lambda(M,g^Y_{\bar{t}}) =0$, so $g^Y_{\bar{t}}$ is scalar flat, 
and  $[0,\bar{t}) \ni t \rightarrow s_{g_t^Y}<0$. We prove first that 
$g_{\bar{t}}^Y$ must be Ricci flat.  

Suppose that $r_{g^Y_{\bar{t}}}\neq 0$. Then there exists 
$\varepsilon>0$ such that $[\bar{t}-\varepsilon , \bar{t}]\subset [0,\bar{t}]$,
and the restricted path $[\bar{t}-\varepsilon , \bar{t}]\ni t \rightarrow 
g_t^Y$ consists of non Einstein Yamabe metrics connecting 
$g_{\bar{t}-\varepsilon}^Y$ to $g_{\bar{t}}^Y$. We consider the dilation
factor $c^2$ in (\ref{minf}) so that 
$f_{c^2g_{\bar{t}}^Y}$ is minimal, and normalize the original path of Yamabe
metrics to $\tilde{g}_t^Y= (\mu_{c^2g_{\bar{t}}^Y}/\mu_{g_{t}^Y})^{\frac{2}{n}} 
g_t^Y$. Then $[0,\bar{t}]\ni t \rightarrow \tilde{g}_{t}^Y$ is a 
differentiable path of $\mu_{c^2g_{\bar{t}}^Y}$ 
volume Yamabe metrics connecting $(\mu_{c^2g_{\bar{t}}^Y}/\mu_{g_{-}})^{\frac{2}
{n}}g_{-}$ and the non Ricci flat zero scalar curvature metric 
$\tilde{g}_{\bar{t}}^Y$ of minimal embedding, 
$[0,\bar{t}) \ni t \rightarrow s_{\tilde{g}_t^Y}<0$, and none of the metrics
$[\bar{t}-\varepsilon,\bar{t}]\ni t \rightarrow 
\tilde{g}_t^Y$ are Einstein, and the embeddings  
$[\bar{t}-\varepsilon,\bar{t}]\rightarrow 
f_{\tilde{g}_t^Y}$ are minimal.
Since the path consists of metrics of fixed volume, by the same 
argument in the first part of the proof of Theorem \ref{thty}, we conclude 
that $(\bar{t}-\varepsilon,\bar{t}) \ni t \rightarrow \lambda(M, \tilde{g}_t^Y)$
is stationary, so constant on the closure of the interval.
Thus, 
$[\bar{t}-\varepsilon,\bar{t}] \ni t \rightarrow \lambda(M, \tilde{g}_t^Y)=
s_{\tilde{g}_{t}^Y}\mu_{\tilde{g}_{t}^Y}= 
s_{\tilde{g}_{\bar{t}-\varepsilon}^Y}\mu_{\tilde{g}_{\bar{t}-\varepsilon}^Y} 
<0$, which contradicts the fact that $\lambda(M,\tilde{g}_{\bar{t}}^Y)=0$.  
Hence, the metric $g_{\bar{t}}^Y$, at which $t \rightarrow \lambda(M,g^Y_t)$ 
vanishes for the first time, must be Ricci flat.

Since the existence of $g'$ suffices to conclude that $M$ is of Kazdan-Warner 
type I, by Theorem \ref{thty}, it follows that 
in the set of scalar flat metrics on $M$, the Ricci flat ones form a
proper subset. If $M$ is orientable, by using any Ricci flat 
metric in the role of $g$ in \cite[\S 3.2, Theorem 2]{sidc}, we see that the
Riemannian metric representative of $w_2(M)$ must be trivial in cohomology, and
so $M$ is then spinnable. 

Since the sectional curvature of an Einstein metric on a 3d 
manifold is constant, so any one such corresponds to a 3d closed space
form, the existence of Einstein metrics of scalar curvature of different signs
forces the condition $n>3$.

If we assume further that $r_{g'}>0$, by Myers' theorem \cite{myer}, the 
fundamental group of $M$ is finite, and its universal cover $\tilde{M}$ is 
compact.  We prove then that it is not possible to have $n=4$. Indeed, if $M$ 
is orientable, then it is spinnable, and by Rohlin's theorem \cite{roh}, 
$\tau(M)= 16 k$ for $k \in \mb{Z}$. Since the Ricci tensor of any metric
can be used to define a Riemannian representative of the even class
$w_2(M)$ \cite[\S 3.2, Theorem 2]{sidc}, there is a spin structure
on $M$ associated to the metric $g_{+}$, and since
$s_{g_{+}}>0$, by  Lichnerowicz's theorem \cite{lic}, 
$$
\hat{A}(M)=0=-\frac{1}{24} p_1(M)=-\frac{1}{8}\tau(M) \, , 
$$     
so $\tau(M)=0$. If $g$ is any metric on $M$, 
and $W_g=W_g^+ + W_g^{-}$ is the Weyl curvature tensor of $g$ 
decomposed into its $\pm 1$ Hodge star $*_g$ eigenspace components, by
the  Chern-Gauss-Bonnet and signature theorems, we have that
$$
\chi(M)={\displaystyle \frac{1}{8\pi^2}\int \left( |W^{+}_g|^2 +|W_g^{-}|^2- 
\frac{1}{2}|r_g -\frac{s_g}{4}g|^2+\frac{1}{24} s_g^2\right) d\mu_g } \, ,
$$
and 
$$
\tau(M)=b_{+}-b_{-}={\displaystyle \frac{1}{12\pi^2}\int \left( |W^{+}_g|^2- 
|W^{-}_g|^2\right) d\mu_g } \, ,
$$
respectively. Hence, if $g$ is Einstein of nonzero scalar curvature, 
$\chi(M)>0$, and 
$$
2\chi(M) \pm \tau(M) ={\displaystyle \frac{1}{4\pi^2}\int \left( 
2|W^{\pm}_g|^2 + \frac{1}{24} s_g^2\right) d\mu_g } > 0 \, .     
$$
We then see that the universal cover $\tilde{M}$ of $M$, which can be
provided with induced metrics of positive Ricci tensor and
Einstein metrics of negative Ricci tensor by lifting those in $M$, 
carries also
scalar flat non Ricci flat and Ricci flat metrics, and being orientable, it is 
spinnable, has even intersection form, and $\tau(\tilde{M})=0$. Since the 
only even unimodular forms with signature zero are the $k$ fold product 
hyperbolic forms  
$k H=k\left( \begin{array}{cc}
0 & 1\\ 1 & 0
\end{array}
\right)$, 
by Freedman's classification theorem in the topological category \cite{free}, 
$\tilde{M}$ is either $\mb{S}^4$ or $\#_{i=1}^k \mb{S}^2 \times \mb{S}^2$, 
with differentiable structure and orientation compatible with those of $M$. 
None of these manifolds can carry both metrics of positive Ricci 
tensor and Einstein metric of negative Ricci tensor. For if any one did, 
since a 4d manifold is Einstein if, and only if, the sectional curvature 
of any plane is equal to the sectional curvature of its orthogonal complement 
\cite{sith}, given a positively oriented normal frame $\{ e_1, e_2, e_3,e_4\}$
for a Ricci flat metric $g_0$, we will have equality of the sectional 
curvatures of
$e_1 \wedge e_2$ and $e_3 \wedge e_4$,
$e_1 \wedge e_3$ and $e_2 \wedge e_4$, and 
$e_1 \wedge e_4$ and $e_2 \wedge e_3$, respectively, and 
$s_{g_0}=0= 4(K^{g_0}(e_1 \wedge e_2)+ K^{g_0}(e_1 \wedge e_3)+
K^{g_0}(e_1 \wedge e_4))$, so at any point either 
the three sectional curvatures $K^{g_0}(e_1 \wedge e_2)$, $K^{g_0}(e_1 
\wedge e_3)$, $K^{g_0}(e_1 \wedge e_4)$ vanish, or otherwise, at least one
is positive and one is negative. 
We consider a finite good cover $\{ U_\alpha\}_{\alpha\in \mc{A}}$ of the 
manifold by convex sets, and positively oriented $g_0$ normal frames 
$\{ e_i^{\alpha}\}$ on $U_{\alpha}$ that we use to trivialize $TU_{\alpha}$. 
 We denote by
$U_{\alpha}^0$ the set of points in $U_{\alpha}$ where  
the sectional curvatures $K^{g_0}(e^{\alpha}_1 \wedge e^{\alpha}_2)$, 
$K^{g_0}(e^{\alpha}_1 \wedge e^{\alpha}_3)$, and $K^{g_0}(e^{\alpha}_1 \wedge 
e^{\alpha}_4)$ are all three equal to zero. 
Since $\tilde{M}$ is closed and simply connected, $g_0$ cannot be flat, and
there exists at least one 
$U_{\alpha}$ such that $U_{\alpha}\setminus U_{\alpha}^0 \neq \varnothing$, 
and at the expense of enlarging the covering, if necessary, we can assume
that $U_{\alpha}^0 = \varnothing$. 
If $K^{g_0}(e^{\alpha}_1\wedge e^{\alpha}_2)>0$, we set 
$J_{\alpha} e^{\alpha}_1=e^{\alpha}_2$, $J_{\alpha}e^{\alpha}_3=
e^{\alpha}_4$, if $K^{g_0}(e^{\alpha}_1\wedge e^{\alpha}_4)>0$, we set
$J_{\alpha}e^{\alpha}_1=e^{\alpha}_4$, $J_{\alpha} e^{\alpha}_2=e^{\alpha}_3$, 
and if $K^{g_0}(e^{\alpha}_1\wedge e^{\alpha}_3)>0$, we set
$J_{\alpha} e^{\alpha}_1=e^{\alpha}_3$, $J_{\alpha} e^{\alpha}_4=e^{\alpha}_2$,
 respectively, to define a tensor $J_{\alpha}$ on $TU_{\alpha}$ such that 
$J_{\alpha}^2= -\BOne\mid_{TU_{ \alpha}}$, and  
which by using the $\mb{S}\mb{O}(4)$ valued transition functions on 
overlapping sets in the covering, can in turn be extended 
to produce a $g_0$ compatible almost complex structure tensor $J_{g_0}$ on the 
entire tangent space of the manifold in question. This excludes from 
consideration the cases of
$\mb{S}^{4}$ and $\#_{i=1}^{k\geq 2}\mb{S}^2\times \mb{S}^2$, which by
the arithmetic genus obstruction, are manifolds that cannot carry 
almost complex structures at all. The remaining case of 
$\mb{S}^2\times \mb{S}^2$ left to consider is a manifold that has a unique
differentiable structure, and which carries a K\"ahler Einstein metric
$(g,J)$ of scalar curvature $s_g=8$ and minimal linear embedding 
$f_g: (\mb{S}^2\times \mb{S}^{2},g)\rightarrow
(\mb{S}^5,\tg)$. The differentiable structure of this manifold cannot be 
the same as that of the homeomorphic manifold $\tilde{M}$ on which 
we have a $J_{g_0}$ invariant Ricci flat metric $g_0$. For suppose
that $J$ and $J_{g_0}$ are in the same orientation class, so 
$(g,J)$ and $(g_{0},J_{g_9})$ are smooth almost Hermtian structures, and they
can be connected to each other by a path $t \rightarrow (g_t,J_t)$ of such.
The smooth path of metrics $[0,t]\rightarrow g_t$ on this manifold starts at 
$g$ and ends at $g_0$. We consider the lift of the path of classes $t\rightarrow
[g_t]$ to a path $t \rightarrow g_t^Y$ of Yamabe metrics of $\mu_{g_t}$ volume
in $[g_t]$, and let $\bar{t}$ be the first time such that
$s_{g_{t}^Y}=0$. Then the restricted path $[0,\bar{t}] \rightarrow \tg_t^Y:=
(\mu_g/\mu_{g_t^Y})^{\frac{1}{2}}g_t^Y$ of constant volume Yamabe metrics 
connects $g$ to the zero scalar curvature metric $\tg_{\bar{t}}^Y$, 
and for some sufficiently small $\varepsilon > 0$, (which depends only
on the positive spectrum of $\Delta^{\tilde{g}_{\bar{t}}^Y}$),  and $t$s in
$(\bar{t}-\varepsilon,\bar{t})$, the paths $t \rightarrow 
\lambda(\mb{S}^2\times \mb{S}^2, \tilde{g}_t^Y)$, 
$t\rightarrow \Phi_{f_{\tilde{g}_t^Y}}(\mb{S}^2 \times \mb{S}^2, 
\tilde{g}_t^Y)$ and $t\rightarrow \Pi_{f_{\tilde{g}_t^Y}}(\mb{S}^2
\times \mb{S}^2, \tilde{g}_t^Y)$ are differentiable, and
$t\rightarrow f_{\tilde{g}_t^Y}$ is a path of minimal embeddings.  
By the argument of earlier, on  
$(\bar{t}-\varepsilon,\bar{t})$ the functional
$\lambda(\mb{S}^2\times \mb{S}^2, \tilde{g}_t^Y)$ is stationary, which 
contradicts the fact that its value at $\bar{t}-\varepsilon$ is positive
while its value at $\bar{t}$ is zero.
If on the other hand, $M$ is not orientable, we use a 2-to-1 
covering $M'\rightarrow M$ by an 
orientable $M'$, which we provide with the lift of the metrics $g'$ and $g_{-}$
to obtain metrics of positive Ricci curvature and Einstein metric of negative 
scalar curvature. By the preliminary result in the orientable case,
no such $M'$ exists.

Finally, suppose that $M^4$ admits an Einstein metric $g$. If $s_g \neq 0$,
we know already that if $s_g>0$, $M$ does not admit an Einstein metric whose
scalar curvature is negative, while if $s_g<0$, $M$ does not even admit metrics
of positive Ricci tensor. In none of these cases, $M$ admits Ricci 
flat metrics. For if otherwise $g_0$ is a Ricci flat metric on $M$, 
if necessary, we pass to a 2-to-1 cover and assume that $M$ is orientable.
Then, as above, we may use $g_0$ to show that $M$ is spinnable and 
carries a compatible almost complex structure. If $g$ is such that
$s_g >0$, with an appropriate spin structure associated to it, Lichnerowicz 
theorem applies, and we end up 
in the situation discussed in the previous paragraph, where the universal 
cover of $M$ is compact, of even intersection form and vanishing signature, and
carries Einstein metrics of positive and zero scalar curvature, 
properties that together makes its existence impossible.
If $s_g<0$ instead, we first observe that by the Cheeger-Gromoll splitting 
theorem, $(M,g_0)$ admits a finite Riemannian covering that is a direct
product of a finite quotient $(\tilde{M},\tilde{g}_0)$ 
of a simply connected Ricci flat manifold 
$(\bar{M},\bar{g}_0)$ 
and a flat $k$-dimensional torus $(T^k,g_k)$, $0\leq k\leq 4$, with 
$\pi_1(M)$ having a subgroup of finite index isomorphic to $\mb{Z}^k$
(the universal cover of $M$ is $\mb{R}^k\times \bar{M}$ with the product
 of the flat metric and $\bar{g}_0$). Hence, the 
metric $g$ can be lifted to an Einstein metric $g'$ on 
$\tilde{M} \times T^k$ with $s_{g'}<0$. Since there are no simply connected
Ricci flat manifolds of dimensions $1, 2$ or $3$, $k$ can be either $0$
or $4$. But neither a K3 surface nor the torus $T^4$ admit an Einstein
metrics of negative scalar curvature, contradicting the existence of
$g'$. If on the other hand the metric $g$ is Ricci flat to begin with,  
the portion of the proof just provided implies then that $M$ does not carry 
Einstein metrics of either positive or negative scalar curvature.    
\qed

\begin{remark}
There are infinitely many examples of 4d manifolds that admit no Einstein 
metric at all, the simplest of them being those constructed on the basis
of Gromov's simplicial volume estimate \cite{grom} (see 
\cite[Theorem 6.47, Example 6.48]{be}). 
And there are homeomorphic minimal surfaces of general type
with an arbitrarily large number of distinct smooth structures of ample 
canonical bundle, hence carrying K\"ahler Einstein metrics of negative
scalar curvature (so a given 4d topological manifold can have an arbitrarily
large number of smooth structures each admitting Einstein metrics).      
It is hard to imagine that the sigma invariants of these smooth
structures on the surface are different. 
\end{remark}

\begin{theorem} \label{am}
Suppose that $[0,1]\ni t \rightarrow f_{g_t}: (M,g_t) \rightarrow (
\mb{S}^{\tn},\tg)$ is a smooth path of differentiable isometric embeddings of 
a smooth path $t\rightarrow g_t$ of smooth Yamabe metrics. 
Then $(f_{g_{\bar{t}}}(M),\tg\mid_{f_{g_{\bar{t}}}(M)})$ has constant sectional 
curvature $c_{g_{\bar{t}}}$ for some $\bar{t}\in [0,1]$ 
if, and only if, $(f_{g_{t}}(M),\tg\mid_{f_{g_{t}}(M)})$ has constant sectional 
curvature $c_{g_t}$ for all $t \in [0,1]$, in which case, the sign of the
constants in the path $[0,1]\ni t\rightarrow c_{g_t}=\frac{1}{n(n-1)}s_{g_t}$
is unique.
\end{theorem}

{\it Proof}. The result holds by definition if $n=1$. If $n=2$, then the
path of scalar curvatures is given by $t \rightarrow s_{g_t}=(8\pi \chi(M)/ 
(\mc{W}(M,[g_t])^{\frac{1}{2}})/(\mc{W}(M,[g_t])/4)^{\frac{1}{2}}=
16\pi \chi(M)/\mc{W}(M,[g_t])$, with 
$\mc{W}(M,[g_t]) \geq \mc{W}(M,[g_k])=4\mu_{g_k}(M)$ where $[g_k]$ is the 
conformal class of a metric $g_k$ of minimal isometric embeddign $f_{g_k}$ 
realizing the sigma invariant of the surface, and so the result clearly holds 
if $\chi(M)\neq 0$, as well as 
if $\chi(M)=0$ since then the paths $[g_t]\rightarrow \mc{W}_{f_{g_t}}(M)
=\mc{W}_{ f_{g_t}}(M)=\mc{W}(M,[g_t])$ and 
$t \rightarrow [g_t]\rightarrow s_{g_t}=\frac{1}{\mu_{g_t}(M)}
(\mc{W}_{f_{g_t}}(M)- \mc{D}_{f_{g_t}}(M))=0$ are both smooth. 
Thus, we assume that $n\geq 3$. 
We prove then that if at $t=\bar{t}$ and for any section we have
that $K^{g_t}(u,v)\mid_{t=\bar{t}}= 
\frac{1}{n(n-1)}s_{g_{\bar{t}}}$, then for all $t$ and for any section 
we have that   $K^{g_t}(u,v)= \frac{1}{n(n-1)} s_{g_t}$. 

We rescale the path of metrics to a path $t \rightarrow g'_t:=(\mu_{ g_{
\bar{t}}}(M)/\mu_{g_{t}}(M))^{\frac{2}{n}}g_t$ of $\mu_{g_{\bar{t}}}(M)$ 
volume Yamabe metrics, and consider the path $t \rightarrow f_{g'_t}$ of their 
isometric embeddings. The hypothesis apply for the rescaled metrics and their
isometric embeddings.    
Suppose that $u,v$ are $g'_{\bar{t}}=g_{\bar{t}}$ orthonormal vectors spanning 
a section in the tangent space at $p'_{\bar{t}}=f_{g'_{\bar{t}}}(p) \in 
f_{g'_{\bar{t}}}(M)$, which we complete to an orthonormal basis of this 
tangent space, and extend by parallel translation to an orthonormal frame 
$\{e_1, e_2, \ldots , e_n\}$ for $\tg\mid_{f_{g'_{\bar{t}}}(M)}$ nearby
such that $e_1\mid_{p'_{\bar{t}}}=u$ and 
$e_2\mid_{p'_{\bar{t}}}=v$. 
We have the decomposition
$T\mb{S}^{\tn}_{\phantom{\tn} \mid_{f_{g'_t}(M)}}=
Tf_{g'_t}(M) \oplus \nu{(f_{g'_t}(M))}$, and 
accordingly, the decomposition
$df_{g'_t}(\partial_t):=T_t=T^{\tau}_t +T_t^{\nu}$ of the  
variational vector field of the embedding $f_{g'_t}$ into tangential and 
normal components. Since by hypothesis $T_t$ varies smoothly with $t$,  
by parallel translation along $T_t$ in the ambient sphere, 
 we extend $\{ e_1, e_2, \ldots, e_n\}$ to
a $\tg\mid_{f_{g'_t}(M)}$ frame
$\{e^t_1, e^t_2, \ldots , e^t_n\}$ nearby 
$p'_t= f_{g'_t}(p)$, for $t$s sufficiently close to $\bar{t}$. Interpreted
intrinsically, the $g'_t$ frame $\{e^t_1, e^t_2, \ldots , e^t_n\}$ 
nearby $p \in M$ so constructed is such that, the section spanned by
$e_1^t\mid_p, e_2^t\mid_p$ in $T_pM$ is compatible with the smoothness 
assumptions of the isometric embeddings $f_{g'_t}$ for $t$s close to $\bar{t}$.
If $h_t= \dot{g}^{'}_t=h^{[g'_t]}:=\frac{(h_t,g'_t)_{g'_t}}{n} g'_t + 
h_t^0$, where $h_t^0$ has $g'_t$ trace $0$, we prove that 
\begin{equation} \label{grt}
\frac{d}{dt}K^{g'_t}(e_1^t,e_2^t)\mid_{t=\bar{t}}=D_{g'_{\bar{t}}}
(K^{g'_t}(e_1^t,e_2^t))(h_{\bar{t}})=D_{g'_{\bar{t}}}(K^{g'_t}(e_1^t,e_2^t))
(h^{[g'_{\bar{t}}]}+ h_{\bar{t}}^0)= 0\, . 
\end{equation}

We have that $K^{g'_t}(e_1^t \mid_p, e_2^t\mid_p)\mid_{t=\bar{t}}=
(1/n(n-1))s_{g'_{\bar{t}}}$, and when the metric varies in the 
conformal direction $h^{[g'_t]}$ at $t=\bar{t}$, the differential is $1/n(n-1)$ 
times the variation of the scalar curvature $s_{g'_t}$ in this direction
at $t=\bar{t}$, which vanishes since $g'_t$ is a Yamabe metric of constant 
volume $\mu_{g'_{\bar{t}}}$. Thus, 
$$
D_{g'_{\bar{t}}}(K^{g'_t}(e_1^t,e_2^t))(h^{[g'_{\bar{t}}]})= 0\, .
$$

For the proof that the remaining portion of the differential vanishes also, it 
is convenient to start with the expression for the sectional curvature in 
extrinsic terms
$$ 
K^{g'_t}(e_1^t,e_2^t)=1+ 
\tg(\alpha_{f_{g'_t}}(e_1^t, e_1^t), \alpha_{f_{g'_t}} 
(e_2^t,e_2^t))- \tg(\alpha_{f_{g'_t}}(e_1^t,e_2^t),\alpha_{f_{g'_t}}(e_1^t,
e_2^t)) 
$$
that derives by Gauss' equation.  The symmetric tensor $h_t^{0}$, of vanishing 
$g'_t$ trace, is a purely tangential two tensor on $T\mb{S}^{\tn}_{\phantom{
\tn} \mid_{ f_{g'_t}(M)}}= Tf_{g'_t}(M) \oplus \nu{(f_{g'_t}(M))}$. We 
use a normal frame of $\nu(f_{g_{\bar{t}}}(M))$ nearby $p'_{\bar{t}}$ to 
extend both, $h^0_{\bar{t}}$ and the frame $\{ e_1, e_2, \ldots, e_n\}$,  
 covariantly constant in the normal directions to a tensor 
$\tilde{h}^0_{\bar{t}}$ and orthonormal vector fields 
$\{ \tilde{e}_1, \tilde{e}_2, \ldots, \tilde{e}_n\}$ that are defined on a 
tubular neighborhood of $f_{g'_{\bar{t}}}(M)$ in $\mb{S}^{\tn}$. 
Since $e_1$ and $e_2$ are $g_{\bar{t}}$ covariantly constant extensions of 
$u$ and $v$, we obtain that
$$
\begin{array}{rcl} 
D_{g'_{\bar{t}}}(K^{g'_t}(e_1^t,e_2^t))
(h_{\bar{t}}^0) & = & 
\tg((\nabla^{\tg})' \tilde{h}^0_{\bar{t}}(e_1,e_1),\alpha_{f_{g'_{\bar{t}}}}(
e_2,e_2))+ \tg((\nabla^{\tg})' \tilde{h}^0_{\bar{t}}(e_2,e_2),\alpha_{f_{g'_{
\bar{t}}}}(e_1,e_1)) \vspace{1mm} \\ & & 
-2\tg((\nabla^{\tg})' \tilde{h}^0_{\bar{t}}(e_1,e_2),\alpha_{f_{g'_t}}(e_1,
e_2)) \, ,  
\end{array}
$$
where $(\nabla^{\tg})' \tilde{h}^0_{\bar{t}}$ stands for the 
variation of the Levi-Civita connection $\nabla^{\tg}$ of the ambient 
space metric $\tg$ as $\tg$ varies in the direction of 
$\tilde{h}^0_{\bar{t}}$,   
which for vector fields $X,Y,Z$ on the sphere may be explicitly computed by the
identity 
$$
\tg((\nabla^{\tg}){'}\tilde{h}^0_{\bar{t}}(X,Y),Z)=\frac{1}{2}\left(
\nabla^{\tg}_X\tilde{h}^0_{\bar{t}}(Y,Z)+\nabla^{\tg}_Y\tilde{h}^0_{\bar{t}}
(X,Z) -\nabla^{\tg}_Z \tilde{h}^0_{\bar{t}}(X,Y)\right) \, .  
$$
Evaluating at $p'_{\bar{t}}$, we obtain that
$$
\begin{array}{rcl} 
D_{g'_{\bar{t}}}(K^{g'_t}(e_1^t,e_2^t))
(h_{\bar{t}}^0) & = & -\frac{1}{2}(  
\nabla^{\tg}_{\alpha_{f_{g'_{\bar{t}}}(e_2,e_2)}}\tilde{h}^0_ {\bar{t}}(
e_1,e_1) + \nabla^{\tg}_{\alpha_{f_{g'_{\bar{t}}}(e_1,e_1)}}\tilde{h}^0_{
\bar{t}}(e_2,e_2)) \\ & & 
+ \nabla^{\tg}_{\alpha_{f_{g'_{\bar{t}}(e_1,e_2)}}}\tilde{h}^0_{\bar{t}}(e_1,
e_2)) \\ 
& =  & 0 \, .  
\end{array}
$$

By (\ref{grt}), the $g'_t$ section $e_1^t \wedge e_2^t$, whose spanning
set may or not be changing with $t$, has curvature that is independent of $t$, 
equal to its 
curvature at $t=\bar{t}$. By the arbitrariness of the initial section
$e_1 \wedge e_2$,  the sectional curvatures of $g'_t$ are all equal to 
$(1/n(n-1))s_{g'_{\bar{t}}}$, which is the sectional curvature of 
$g'_{\bar{t}}$. The desired
result for the sectional curvatures of $g_t$ follows since
$s_{g'_t}=(\mu_{g_{\bar{t}}}(M)/\mu_{g_t}(M))^{-\frac{2}{n}}s_{g_t}$.
\qed 

\section{The $\sigma$ invariant of manifolds with standard model geometries}
\subsection{Standard geometric models}
In dimension $3$, the stabilizer subgroups of any geometric model 
$({\rm Isom}(X,g),X)$ have identity component subgroups that can be either
$\mb{S}\mb{O}(3)$, $\mb{S}\mb{O}(2)$, or $\mb{S}\mb{O}(1)$,
of dimensions $3$, $1$ and $0$, respectively. In the first case,
the model $X$ is either the $3$-sphere, Euclidean $3$-space, or hyperbolic
$3$-space. In the second case, $X$ fibers over a two 
dimensional model geometry in a way that is invariant under ${\rm Isom}(X,g)$, 
which is a group of dimension $4$, and the metric connection orthogonal to the 
fibers has curvature $0$ or $1$, 
yielding four other distinct models that are line bundles over two 
dimensional ones. In the last case, we obtain a solvmanifold model
$X$ that fibers over the line, and which itself identifies with
the identity component of the group ${\rm Isom}(X,g)$ of dimension $3$. 
Each of these eight models has an essentially
unique differentiable structure. We describe their version in
full dimensional generality, making the standard choice of 
differentiable structure then being considered when $n>3$. 

\subsubsection{Elliptic, Euclidean and hyperbolic geometric models}
The homogeneous manifold models leading to
elliptic, Euclidean and hyperbolic geometries are 
 the standard sphere $(\mb{S}^n,
g_ {\mb{S}^n})$, $g_{\mb{S}^n}$ a metric of sectional curvature 
$1-\delta_{n,1}$, $\mb{S}^n \cong \mb{O}(n+1)/\mb{O}(n)$, the Euclidean space 
$(\mb{R}^n, g_{\mb{R}^n})$, $g_{\mb{R}^n}$ its standard flat metric, 
$\mb{R}^n \cong \mb{O}(n)\ltimes \mb{R}^n/\mb{O}(n)\times 
\{ 0\}$, and the hyperbolic space $(\mb{H}^n,g_{\mb{H}^n})$, 
$g_{\mb{H}^n}$ a metric of sectional curvature $-1+\delta_{n,1}$,
$\mb{H}^n=\mb{O}^+(1,n)/ \mb{O}(n)$. The isometry group of these models is of 
the largest possible dimension $n(n+1)/2$, and the stabilizer of a point is 
$\mb{O}(n)$. The latter two are simply connected, as is the first if 
$n\geq 2$. If $n=1$, the universal cover of 
$(\mb{S}^1,g_{\mb{S}^1})$ with the lifted metric is 
$(\mb{R}^1, \| \, \|^2)$, and for dimensional reasons we have the isometric 
identifications $(\mb{R}^1, \| \, \|^2) \cong (\mb{R}^1,g_{\mb{R}^1})
\cong (\mb{H}^1,g_{\mb{H}^1})$ producing the then single model geometry, with 
the quotient $\mb{R}^1/\mb{Z}=(\mb{S}^1,{g_{\mb{S}^1}})$ inheriting a trivial
CR structure from the complex structure of $\mb{R}^2\subset \mb{S}^2\setminus
\{(0,0,1)\}$. 

\subsubsection{Trivial line bundles over nonzero constant sectional curvature 
lower dimensional models}
We have the special cases of homogeneous  
manifolds $(\mb{S}^{n-1}\times \mb{R}, g_{\mb{S}^{n-1}}+dr^2)$ 
and $(\mb{H}^{n-1}\times \mb{R}, g_{\mb{H}^{n-1}} +dr^2)$ of 
isometry groups $\mb{O}(n) \times \mb{R}$,   
and $\mb{O}^{+}(1,n-1)\times \mb{R}$, respectively, which act transitively
on the corresponding manifolds with compact point stabilizers. If $n>2$, these 
cases can be characterized as trivial line bundles over a homogeneous base of 
nonzero constant sectional curvature. If $n=2$, the universal cover of 
$\mb{S}^1 \times \mb{R}$ is $\mb{R}\times \mb{R}$ with a flat metric on it.  

\subsubsection{Lie group models}
Any Lie group $(M,g)$ with a left invariant metric $g$ is a homogeneous 
manifold,
and if simply connected, it yields a model geometry. In these cases, the 
group acts isometrically on itself by left multiplication, and we obtain an 
embedding $M \hookrightarrow {\rm Isom}\, (M,g)$.  
Three distinguished examples of this type are:

\noindent (a) The Heisenberg group: 
The group of linear fractional transformations of the 
unit ball $B=\{ z \in \mb{C}^{1+n}: \; 1-\sum_{k=1}^{n+1} |z_k|^2>0\}$
in $\mb{C}^{n+1}$ is given by $\mb{S}\mb{U}(1,n+1)$, and 
the mapping
$$
\begin{array}{rcl}
\mb{C}^{2+n}  & \rightarrow & \mb{C}^{1+n} \\
 (Z_0, \ldots, Z_{n+1}) & \mapsto & \frac{1}{Z_0}(Z_1, \ldots, Z_{n+1}) 
\end{array}
$$
establishes a correspondence between the set 
$$
\{ Z : \; | Z_0|^2- \sum_{k=1}^{n+1} | Z_k|^2 > 0 \}
$$
and $B$, where by definition, $\mb{S}\mb{U}(1,n+1)$ is the group of 
linear transformations on $\mb{C}^{n+2}$ that preserve the quadratic form 
defining the set, and thus defining its action on $B$. 
Since $\mb{S}^{2n+1}=\partial B$, this defines also the action of the group 
on the sphere.

The Siegel domain $D=\{ (z_1,z)\in \mb{C}^{1+n}: \; {\rm Im}\, z_1 > | 
z |^2\}$ is identified with the set
$$
\{ Z : \; \frac{1}{2i}( \overline{Z}_0 Z_1 -Z_{0}\overline{Z}_1) - 
\sum_{k=1}^{n+1} | Z_k|^2 > 0 \}\, , 
$$
and the action of $\mb{S}\mb{U}(1,n+1)$ on it thus defines its action on
$D$. The linear fractional transformation 
$$
T(z_1,z)=\frac{1}{i(z_1-1)}(z_1+1,z) 
$$
maps $B$ onto $D$, with the north pole $(1,0,\ldots, 0)
\in \mb{S}^{2n+1}\subset \mb{C}^{1+n}$ being sent to $\infty$.

The Heisenberg group $H_{n}$ is the boundary of $D\subset \mb{C}^{1+n}$ with 
its structure of a nilpotent Lie group given by  
$$
(z_1,z)\cdot (w_1, w)=( z_1+w_1+2i\< z, w\>, z+w) \, .   
$$
This group thus have a standard CR structure $(\mc{D},J)$ induced by the 
complex structure on the ambient space $\mb{C}^{1+n}$. Its 
dimension is $n'=2n+1$, and the depiction shows it as
the quotient of $\mb{S}\mb{U}(1,n+1)$ by the isotropy group of the point at 
$\infty$.  As a group, it is isomorphic to $\mb{R} \times \mb{C}^{n}=\mb{R} 
\times \mb{R}^{2n}$ with group structure 
$$
(\theta,z) \cdot (\theta',z') =(\theta+\theta'- 2{\rm Im}\< z , z'\>,
 z+z') \, ,  
$$
and the identifications between the two models is explicitly given by  
$$
(\theta,z) \longleftrightarrow (\theta+i |z|^2, z ) \, .
$$
In the latter model, 
the identity element is 
$(\theta, z)=(0,0)$, and the group inverse of $(\theta, z)$ is the
element $(-\theta, -z)$, so ${\rm Ad}_{(\theta,z)}(a,u)=(a+
4{\rm Im } \< u,z\>,u)$, and  
$$
[(b,v), (a,u)] ={\rm ad}_{(b,v)}(a,u)=(4{\rm Im}\< u,v\>,0)
\, , 
$$
hence, if $(a,u)=(a,x+iy) \in \mb{R}\times \mb{C}^{n}$ and 
$M_{(a,u)}=M_{a,x,y}$ is the $(n+2)\times (n+2)$ matrix all of whose 
entries are zero except for those in its first row and last column that are 
equal to the vectors $(0,x,a)$ and $(a,y,0)$, respectively, the map
$(a,u) \rightarrow 2M_{(a,u)}$ identifies the Lie algebra 
$\mb{R}\times \mb{C}^{n}$ of $H_{n}$ with the nilpotent algebra of matrices 
$\mathfrak{h}_{n'}=\{ M_{a,x,y} | \; (a,x,y)\in \mb{R}\times \mb{R}^{2n}\}$, 
the exponential map $\mf{h}_{n'} \stackrel{\rm exp}{\rightarrow} H_{n}$ is 
onto, $\mf{h}_{n'}$ is the universal covering space of $H_{n}$,  
and we have that $e^{M_{(\theta, z)}}e^{M_{(\theta',z')}}=
e^{M_{(\theta, z)\cdot (\theta',z')}}$, showing the 
correspondence between ${\rm Aut}(\mf{h}_{n'})$ 
and ${\rm Aut}(H_{n})$. This model is thus the total space of a line bundle
over Euclidean $\mb{R}^{2n}$ equipped with a connection of constant 
curvature $1$ and homogeneous metric $g$ such that, the connection is 
Riemannian, the metric on horizontal planes is the pullback of the Euclidean
metric on $\mb{R}^{2n}$, and the metric is invariant under the action of
the line in the vertical directions. ${\rm Isom}(H_n,g)$ is generated
by the lift of isometries on the base, and vertical translations.   

\noindent (b) A group $G$ is solvable if it has a composition series 
$\{ G_i\}$ all of whose factor groups $G_{i+1}/G_i$ are Abelian. This can be
characterized in terms of the derived series of the Lie algebra $\mf{g}$, whose
$k$th element is $\mc{D}^k \mf{g}=[\mc{D}^{k-1}\mf{g},\mc{D}^{k-1}\mf{g}]$,
with $\mc{D}^1 \mf{g}=[\mf{g},\mf{g}]$. The algebra $\mf{g}$ is solvable 
if $\mc{D}^k \mf{g}=0$ for some $k$. And solvable Lie groups correspond to
solvable Lie algebras.

If $n=3$, there is a geometric model of this type that does not arise in any 
of the forms described earlier, neither commutative nor nilpotent:  The 
solvable group $S_2=\mb{R}^2 \rtimes \mb{R}^{\times}$, where 
$t\in \mb{R}^{\times}$ acts on $(x_1,x_2)\in \mb{R}^2$ by 
$$
t\cdot (x_1, x_2)=(tx_1, t^{-1}x_2) \, , 
$$
the identity is $(0,1)$, and $((x_1,x_2),t)^{-1}=(-t^{-1}x_1,-tx_2),
t^{-1})$, so ${\rm Ad}_{(x,t)}((v_1,v_2),s)=((tv_1,t^{-1}v_2),s)$ and
$$
[(w,s),(v,s')]={\rm ad}(w,s)(v,s')=((sv_1,-sv_2),0)\, . 
$$ 
The derived central series of subalgebras 
$\mc{D}^k \mf{s}_2=[\mc{D}^{k-1}\mf{s}_2, \mc{D}^{k-1}\mf{s}_2]$ becomes 
trivial after $k=2$ steps. The isometry group of this model has trivial
point stabilizer, so the model itself is a Lie group given by a split
extension $0 \rightarrow \mb{R}^2 \rightarrow S_2 \rightarrow \mb{R} \rightarrow
1$.  If the group is provided with a left invariant metric $g$ so the 
eigenspaces of the action of $t$ on $\mb{R}^2$ are orthogonal, and the fibers 
are orthogonal to $0 \times \mb{R}$, we have that 
${\rm Isom}(S_2,g)= S_2 \rtimes (\mb{Z}/2\mb{Z})^3$, with the finite factors in
the semidirect product corresponding
to reflections in the direction of $\mb{R}$ and in the direction of the two
eigenspaces in $\mb{R}^2$.

\noindent (c) The universal cover $\widetilde{\mb{P}\mb{S}\mb{L}}(2,\mb{R})$ 
of $\mb{P}\mb{S}\mb{L}(2,\mb{R})=\mb{S}\mb{L}(2,\mb{R})/\{ \BOne, -\BOne\}$,
which has the structure of a line bundle 
$\mb{R} \hookrightarrow \widetilde{\mb{P}\mb{S}\mb{L}}(2,\mb{R}) \rightarrow
\mb{H}^2$ over hyperbolic space provided with a connection of nonzero
constant curvature, and whose total space is equipped with the natural 
homogeneous metric $g$ given by the homogeneous metrics on base and fiber, with
fibers orthogonal to the horizontal planes. The action of
$\mb{P}\mb{S}\mb{L}(2,\mb{R})$ on ${\mb H}^2$ induces a faithful action on 
$S^1(T{\mb H}^2)=\mb{S}\mb{L}(2,\mb{R})/\{ \BOne, -\BOne\}$,
thus we obtain a natural identification of 
$\widetilde{\mb{P}\mb{S}\mb{L}}(2,\mb{R})$ with the indicated line bundle.
We have that 
${\rm Isom}(\widetilde{\mb{P}\mb{S}\mb{L}}(2,\mb{R}),g) \cong 
(\widetilde{\mb{P}\mb{S}\mb{L}}(2,\mb{R})\times \widetilde{\mb{O}}(2))/\mb{Z}$,
where $\mb{O}(2)) \subset \mb{P}\mb{S}\mb{L}(2,\mb{R})$ is the stabilizer of a
point and $\pi_1(\mb{P}\mb{S}\mb{L}(2,\mb{R})) \cong
\pi_1(\mb{S}\mb{O}(2)) \cong \mb{Z}$. 

\begin{remark}
Flat $\mb{R}$ geometrizes $1$d manifolds, the homogeneous group 
$\mb{S}\mb{O}(2)$ excluded as it is not simply connected.  Elliptic, 
Euclidean and hyperbolic models suffice for the geometrization of surfaces, 
the solvable Lie group model $\mb{R}\rtimes \mb{R}^{\times}$  
excluded from consideration since 
for any discrete subgroup $\Gamma$ that acts on it without fixed points,
the quotient manifold $\mb{R}\rtimes \mb{R}^{\times}/\Gamma$ 
does not have finite area. 
\end{remark}

\subsection{Manifolds with standard model geometries: 3d manifolds}
The conjectural picture formulated by Thurston in 1982 in the TGT summarized 
the then situation of 3d manifolds of constant curvature, mostly developed in 
the previous part of the XXth century. In 1910, Bieberbach 
classified the discrete groups of rigid motions of $\mb{R}^n$
with compact fundamental domain, proving that there are finitely
many of them. In 1925 Hopf classified compact 3d manifolds of constant 
positive curvature, with the Lens spaces subclass classified in 1935 by 
Reidemeister using his torsion invariant (that depends on a choice of a
representation of its fundamental group). Few examples of manifolds 
of constant negative curvature, the most common type of manifolds in the TGT, 
were known until Thurston's work in the 1970s, although the subject had been
introduced by Poincar\'e in the 1900s. In 1942 Preissmann proved 
that any nontrivial Abelian subgroup of the fundamental group of a closed
$(M^n,g)$ of strictly negative curvature must be free cyclic, but the 
otherwise dormant subject was revitalized in the 1960s \& 1970s by Ahlfords,
Mostow, Sullivan and others, who established a correspondence between 
appropriate deformations of the fundamental group
and finite area hyperbolic metrics on the quotient of its domain of continuity
by it, so the hyperbolic manifold will be one of finitely many cusps. 
In 1975, Riley, using representations of a knot group $\pi_1(\mb{S}^3 \setminus
K)$, produced examples of knots whose complement can be
given the structure of a complete hyperbolic manifold of finite volume, 
and all of these developments inspired the work of Thurston on knot 
complements, who then 
proved that any prime manifold that contains a 2-sided incompressible surface 
that is not the 2-sphere carries a hyperbolic structure if, and only if, it is 
homotopically atoroidal, the earlier alluded bulk of his geometrization 
conjecture, so if such a manifold is closed, then it is hyperbolic. 
The complete proof of the TGT by Perelman  
describes the possible collapse of 
2-spheres in a gauge modified Ricci flow of Hamilton (1982), 
which makes of it a gradient flow, and then shows how these
singularities can be eliminated by an appropriate ``surgery'' to continue the
flow until the desingularized manifold reaches appropriately geometrized 
identifiable limits of three kinds, two with compact models carrying metrics of
positive scalar curvature, while the third admitting a thick-thin decomposition
with hyperbolic thick part, and thin part modeled by the 
remaining geometries after cutting it by tori.

The various advances leading to Schoen's notion of $\sigma(M^{n\geq 3})$ in
1989 begun later but overlap a bit with the historical period above.   
The Yamabe problem for $n\geq 3$ manifolds was originally formulated and 
analyzed 
by Yamabe in 1960, and fully resolved in significant additional installments 
by Trudinger (1968), Aubin (1976) and Schoen (1984), respectively.
The Kazdan-Warner trichotomy of manifolds of these dimensions appeared in
1975,  
while the Nash isometric embedding theorem dates to 1956, and the  
Palais isotopic deformation theorem to 1960, but their use in the
study of the  total scalar curvature in terms of the interrelated extrinsic 
functionals in (\ref{yafu}), conducive to the type of results illustrated in 
\S2 here, is more recent. The remarkable computation of the
rough Laplacian of the second fundamental form of an immersion by Simons 
in 1968 had been used primarily to analize minimal immersion into the 
sphere, drawing little connections with the induced metric on the immersed
submanifold. In 1973, Riley computed the variation
of functionals of density of the form $f(S_1,\ldots, S_n)$, $S_i$ the
ith elementary symmetric function of the principal curvatures of a hypersurface
in $n+1$ space form of curvature $c$, which for dimensional reasons, can be
used easily to find the variation of $\Pi_{f_g}(M)$ for a surface $M$ in 
$\mb{S}^3$. But the variations of the extrinsic functionals in (\ref{yafu}),  
under sufficiently regular deformations of the embedding into an 
arbitrary fixed background, are from as recently as 2004, from which the 
analysis
of the total scalar curvature using the three of them together takes off.
This analysis had been carried out by 
many (and continues to be) using an intrinsic approach, a hard setting when 
considering its 
deformations on a fixed differentiable manifold with a large space of conformal 
classes, which hides the transformation rules of variations in 
a given conformal class of the extrinsic functionals that combine to the 
intrinsic one,  of importance each in our approach,
and which with our  extension of the notion to manifolds of dimensions one and 
two also, have properties that fit in a dimension independent unifying frame.
The splitting theorem of Cheeger \& Gromoll, which 
serves significantly to distinguish properties of the 
fundamental groups of geometrized manifolds of contractible models and 
vanishing sigma invariant that are of Kazdan-Warner type II and III, dates 
all the way back to 1971. 

We apply the technical results we have derived by the analysis of the alluded 
extrinsic functionals to either recall, if already done elsewhere, 
or to determine the Kazdan-Warner type and to compute the sigma invariant
of any 3d closed manifold $M$ with Thurston's geometric model. When  
this invariant is not realized, we identify sequence(s) of 
conformal classes $[g_j^Y]$ of Yamabe representative $g_j^Y$ such that
$\lambda(M,[g_j^Y])$ converges to it.

\subsubsection{Elliptic manifolds}
An elliptic manifold of dimension $n>1$ is finitely covered by the $n$-sphere
and is of the $M=\mb{S}^n/\Gamma_M$ where $\Gamma_M$ is a finite subgroup of 
$\mb{O}(n+1)$ acting freely on $\mb{S}^n$. If $n$ is odd,
$\Gamma_M$ is in fact contained in $\mb{S}\mb{O}(n+1)$, and $M$ is orientable.
Any 3d manifold $M$ of finite fundamental group admits
an elliptic model, and the $\sigma$ invariant and optimal volume
of such can had been explicitly computed in terms of the order  
$|\Gamma_M|$ of $\Gamma_M$ \cite[Theorem 6]{sim7}, the argument for 
doing a recasting of
a continuity argument we used earlier for the computation of the invariant and 
optimal volume of real, complex, and quaternionic projective spaces of any 
dimension \cite[Theorem 5]{sim4}, \cite[Theorem 6]{sim8}. We outline here how
to extend this result to standard Lens spaces of any dimension, bypassing the 
reproduction of the proof.

We let $(L(p, q_0, \ldots, q_n),g_{\Gamma_L})= (\mb{S}^{2n+1}/\Gamma_L,g_{
\Gamma_L})$ be the Lens space in the base of the Riemannian submersion
\begin{equation} \label{eq9} 
\begin{array}{ccc}
\mb{Z}/p\mb{Z} & \hookrightarrow  & (\mb{S}^{2n+1}(r_{L}),g)  \\
& & \downarrow \pi_L  \\ 
& & L(p,q_0,\ldots, q_n),g_{\Gamma_L})  
\end{array}\, , 
\end{equation}
where the action action of $\Gamma_L=\mb{Z}/p\mb{Z}$ on the sphere of radius
$r_L$, $r_L^2=|\Gamma_L|^{\frac{2n+1}{2}}$, is given by 
$(z_0, \ldots ,z_n)\rightarrow (\gamma^{q_0}z_0, \ldots, \gamma^{q_n}z_n)$,   
$\gamma=e^{\frac{2\pi i}{p}}$.
We use elements of the space  
$\mc{S}^{inv}_{\Gamma_L}$ of 
$\Gamma_L$ invariant spherical harmonics homogeneous of degree $|\Gamma_L|$
to realize the metric $g_{\Gamma_L}$ by a minimal isometric embedding
\begin{equation} \label{eq28}
f_{g_{\Gamma_L}}: (L,g_{\Gamma_L}) \rightarrow (\mb{S}^{d^{inv}_{\Gamma_L}},g) 
\hookrightarrow (\mb{R}^{d^{inv}_{\Gamma}+1},\| \, \|^2)
\end{equation}
where $d_{\Gamma_L}^{inv}$ is the dimension of $\mc{S}^{inv}_{\Gamma_L}$.
We then have that
\begin{equation} \label{eq29} 
\begin{array}{rcl}
\mu_{g_{\Gamma_L}}(f_{g_{\Gamma_L}}(L)) & = & 
\omega_{2n+1}\, r_L^{2n+1} \, , 
\vspace{1mm}\\
\| \alpha_{f_{g_{\Gamma_L}}} \|^2 & = & (2n+1)2n\left(1-  
1/\left( |\Gamma_L|^{\frac{2}{2n+1}}r^2_L\right) \right)\, , \\
s_{g_{\Gamma_L}} & = &  (2n+1)2n/\left( |\Gamma_L|^{\frac{2}{2n+1}}r^2_L\right) \, ,  
\end{array}
\end{equation}
respectively. Geometrically, $f_{g_{\Gamma_L}}$ makes $f_{g_{\Gamma_L}}(L)$
and $\mb{S}^{2n+1}(r_L)$ manifolds of the same volume, while relating their
sectional curvatures by the factor $1/|\Gamma_L|^{\frac{2}{2n+1}}$.

\begin{theorem}
If $L:=L(p,q_0,\ldots, q_n)$, we have that
$$
\sigma(L)=\lambda(L, [g_{\Gamma_L}]) = \frac{(2n+1)2n 
\omega_{2n+1}^{\frac{2}{2n+1}}}{p^{\frac{2}{2n+1}}} \, ,
$$
and among all Yamabe metrics $g'$ on $L$ of minimal isometric embedding
$f_{g'}$ and positive scalar curvature $s_{g'}$, $g_{\Gamma_L}$ has the 
largest volume and 
$$
\begin{array}{rcl}
\mc{W}_{f_{g'}}(L)=\mc{W}(L,[g']) & \leq &  
\mc{W}_{f_{g_{\Gamma_L}}}(L)=\mc{W}(L,[g_{\Gamma_L}])=(2n+1)^2 \omega_{2n+1}
r_L^{2n+1}\, , \\ 
\mc{D}_{f_{g'}}(L)=\mc{D}(L,[g']) & \leq &  
\mc{D}_{f_{g_{\Gamma_L}}}(L)=\mc{D}(L,[g_{\Gamma_L}])=(2n+1)(1-2n/|\Gamma_L|^{
\frac{2}{2n+1}}r_L^2)
\omega_{2n+1}r_L^{2n+1} \, .
\end{array}
$$
If $L'=
L(p';q'_1, \ldots, q'_n)$ is such that $p=p'$, $L$ is diffeomorphic to
$L'$ if, and only if, $\mc{S}^{inv}_{\Gamma_L}=\mc{S}^{inv}_{\Gamma_{L'}}$.
\end{theorem}

\begin{remark}
By \cite[\S 5]{sull}, any differential graded algebra $(A,d_A)$ with 
$H^0(A,d_A)=\mb{R}$ has a unique minimal model, 
so the differential algebra of forms of a closed manifold $M$ has a unique 
minimal model. Even in the nonsimply conneced case, this is an invariant of 
the rational homotopy type of $M$. As a rational homotopy sphere, 
the minimal model $\Lambda^{L(p;q_0,\ldots, q_n)}(z^L_{2n+1})$, $dz^L_{2n+1}
=0$, of $L(p;q_0, \ldots, q_n)$ is isomorphic to the minimal model
$\Lambda^{\mb{S}^{2n+1}}(z_{2n+1})$, $dz_{2n+1}=0$ of $\mb{S}^{2n+1}$, 
and the homomorphism of these algebras induced by the map $\pi_L$ in 
(\ref{eq9}) is, at the level of generators,  given by
$m(\pi_L)(z^L_{2n+1}) = cz_{2n+1}$ for some $0\neq c \in \mb{Q}$, naturally 
taken as $1$ if we think of $z^L_{2n+1}$ as the volume form of 
$f_{g_{\Gamma_L}}(L(p;q_0,\ldots,q_n))$, or $1/p$ if 
we think of it as the volume form of $L(p;q_0,\ldots,q_n)$ itself.
Naturally (as that was not its purpose), this model does capture the 
diffeomorphism type of the Lens space, which depends on the $q_i$s, or even 
its integral cohomology. We have 
a remnant $\mb{S}^1=\mb{S}^1/(\mb{Z}/p)$ free action that yields a fiber
projection $L(p;q_0,\ldots, q_n) \stackrel{\pi_p}{\rightarrow }
\mb{P}^n(\mb{C})$ compatible with the fiber bundle projection 
$\mb{S}^{2n+1}\stackrel{\pi^{\mb{P}}}{\rightarrow} \mb{P}^n(\mb{C})$ in
that $\pi^{\mb{P}}=\pi_p \circ \pi_L$, and we may then 
use the nonminimal model 
$\Lambda^{L(p;q_0,\ldots, q_n)}_{\mb{Z}}(u_1, v_2,u_{2n+1},z^L_{2n+1})$,  
$du_1=pv_2$, $dv_2=0$, $du_{2n+1}=v_2^{n+1}$, $dz^L_{2n+1}=0$ of correct 
integer cohomology, 
and relate it to the reduction mod $p$ homomorphism in cohomology 
induced by $\pi_p$ from the minimal model $\Lambda^{\mb{P}^n(\mb{C})}(x_2,
y_{2n+1})$, $dx_2=0$, $dy_{2n+1}=x_2^{n+1}$, of the projective space, 
which at the level of generators is then given by 
$m(\pi_p)(y_{2n+1})=0$, $m(\pi_p)(x_2)=v_2$. With the natural definitions on
these nonminimal models, we obtain homomorphism such that
$m(\pi^{\mb{P}})= m(\pi_L) \circ m(\pi_p)$, but are still unable to 
distinguish nondiffeomorphic but homotopy equivalent Lens
spaces, for instance, $L(7;1,\ldots, 1, 1)$ and $L(7;1,\ldots, 1,2)$. 
It is conceivable that the natural representation of $\pi_1(L(p;q_0, 
\ldots, q_n))$ on $\mb{G}\mb{L}(\mc{S}_{\Gamma_{L}}^{inv})$ could be used
to produce a geometric realization of the Reidemeister torsion of the Lens 
space and associated models for its graded algebra of forms and that of
$\mb{P}^n(\mb{C})$, so that the homomorphism induced by 
the remnant action $\pi_p$ captures the winding of the components of a 
$\pi^{\mb{P}}$ orbit point characterizing $L$, 
even if the real cohomology of this model for $L$ vanishes.   
\end{remark}

\subsubsection{Products of standard spheres: The $\mb{R}\times \mb{S}^{n-1}$
model geometry}
For $n\geq 2$ and $1\leq k < n$, we consider the manifold $M^n_k=
\mb{S}^k\times \mb{S}^{n-k}$ with the standard smooth product structure of 
the factors, which is unique if $n\leq 4$. 
$M^n_k$ is geometrically modeled by the product of 
the geometric models of the factors. 
We denote by $g_{\mb{S}^{n,k}}$ the product metric on  
$\mb{S}^{n,k}:=\mb{S}^k(\sqrt{k/n})\times \mb{S}^{n-k}(\sqrt{n-k/n})$. It is 
a Yamabe metric in its conformal class. 

If $n=2$, we have that $\sigma(M^2_1)=0$, with its canonical realizer being
the Clifford torus $(\mb{S}^{2,1},g_{\mb{S}^{2,1}})$ \cite[\S 2]{sim7}.
If on the other hand $2\leq k \leq n-2$, 
if $[g]$ is a conformal class of metrics on $M^n_k$, then 
$\lambda(M^n_k,[g]) \leq  \lambda(\mb{S}^{n,k},[g_{\mb{S}^{n,k}}]):=
\sigma(M^n_k)$, 
and the equality is achieved if, and only if, $[g]=[g_{\mb{S}^{n,k}}]$; the
realizing class carries an Einstein metric in it if, and only if, $2k=n$
\cite[Theorem 2, Corollary 3]{sim4}. We thus obtain an infinite number of 
$M^n_k$s whose $\sigma$ invariant is realized by a conformal class with 
no Einstein metric representative in it. Notice that by Theorem \ref{th2}, 
$M^4_2$ does not carry Einstein metrics of nonpositive 
scalar curvature, and any scalar flat metric on it, which do exists,
is not Ricci flat.

The remaining case to consider is relevant to the model geometry 
$\mb{R}\times \mb{S}^{n-1}$. Since ${\rm Isom}\,( \mb{R}\times \mb{S}^{n-1}) =
{\rm Isom}\, \mb{R}\times {\rm Isom}\, \mb{S}^{n-1}=(\mb{R}\rtimes \mb{O}(1))
\times \mb{O}(n)$, if $\Gamma$ is either $\< T^{
\mb{R}}_{x_0}\times \BOne_{\mb{S}^{n-1}}\>$ or
$\< T^{ \mb{R}}_{x_0}\times \varphi_{\mb{S}^{n-1}}\>$, $\varphi$ an
orientation reversing diffeomorphism of $\mb{S}^{n-1}$, then the quotient 
manifold  $(\mb{R}\times \mb{S}^{n-1})/\Gamma$ is, correspondingly, the trivial 
sphere bundle $\mb{S}^1 \times \mb{S}^{n-1}= M^n_1 \cong M^n_{n-1}$, or the 
nontrivial one  $\mb{S}^1 \widehat{\times}\mb{S}^{n-1}:= \widehat{M}^n_1$, 
and we have that $\sigma(M^n_1)=\sigma(\mb{S}^n)= \sigma 
(\widehat{M}^n_1)$, with the invariant not achieved by any conformal class in 
either case. The result for $M^n_1$ was proved by Schoen in his precursory
work on the subject \cite{sc2}, and his argument can be modified readily to 
draw the conclusion for $\widehat{M}^n_1$ as well. We obtain an additional 
example of
manifold with this geometric model by taking $\Gamma$ to be
the group $\<T^{\mb{R}}_{x_0}\times A_{\mb{S}^{n-1}}\>$, in which case the 
quotient is $\mb{S}^1 \times \mb{P}^{n-1}(\mb{R})$, 
and again, Schoen's original argument can be adjusted to show that 
$\sigma(\mb{S}^1 \times \mb{P}^{n-1}(\mb{R})) =\sigma(\mb{S}^n)/2^{\frac{2}{n}}
$, without it being achieved by any class. On the other hand,
$\mb{P}^n(\mb{R})\# \mb{P}^n(\mb{R})$ may be constructed as the quotient
of $\mb{S}^{n-1}\times \mb{S}^1$ by the equivalence relation 
$(x,z) \sim (A(x),\bar{z})$, where the conjugation is that on
$\mb{U}(1)\cong \mb{S}^1$ along a great circle connecting the antipodal points
$x$ and $A(x)$. We have that   
$\sigma (\mb{P}^{n}(\mb{R})\#\mb{P}^{n} (\mb{R})) \geq \min\{
\sigma (\mb{P}^{n}(\mb{R})), \sigma(\mb{P}^{n}(\mb{R}))\}$ \cite{koba},  
while using the nontrivial Riemannian double covering map that  
$\sigma (\mb{P}^{n}(\mb{R})\#\mb{P}^{n} (\mb{R})) \leq \sigma( \mb{S}^1\times
\mb{P}^{n-1}(\mb{R}))=    
\sigma(\mb{S}^n)/ 2^{\frac{2}{n}}$ \cite{au3},  so
$\sigma (\mb{P}^{n}(\mb{R})\#\mb{P}^{n} (\mb{R}))=
\sigma (\mb{P}^{n}(\mb{R}))=\sigma(\mb{S}^n)/2^{\frac{2}{n}}$, with the 
invariant again not achieved by any class, as in the previous cases.    
When $n=3$, the exhibited four manifolds constitute all the 
3d manifolds of model geometry $\mb{R}\times \mb{S}^2$, the first and last
orientable, while the second and third not so \cite{brne,akne}.  
 
\subsubsection{Euclidean manifolds}
Every closed Euclidean $n$-manifold is finitely covered by the $n$ torus,
and so they correspond to $n$-dimensional crystallographic groups 
acting freely on $\mb{R}^n$. Thus, $M=\mb{R}^n/\pi_1(M)$ where $\pi_1(M)$ is
a torsion free discrete subgroup of ${\rm Isom}(\mb{R}^n, g_{\mb{R}^n})$
that constains a subgroup of finite index isomorphic
to $\mb{Z}^n$. These are manifolds of Kazdan-Warner type II, and so they carry
a (Ricci) flat metric $g$ of minimal isometric
embedding $f_g$ that is optimal in that $\mc{W}_{f_g}(M) \leq 
\mc{W}_{f_{g'}}(M)$ for any other (Ricci) flat $g'$, 
and which provides $M$ with a canonical volume  
\cite[Theorem 2]{sim7}. 
We have used this characterization to show explicitly that $T^n$  
and any of the ten 3d closed Euclidean 
manifolds are of Kazdan-Warner type II, computing in each case their
canonical volume $\mu_g(M)$ \cite[Theorems 3 \& 5]{sim7}. 

\subsubsection{{\rm 3d} manifolds with nilgeometry}
The geometry provided by the 2-step nilpotent Heisenberg group 
$H_n = \mb{C}^n \times \mb{R}$ is that of a line bundle $\mb{R} 
\hookrightarrow H_n \stackrel{\pi}{ \rightarrow } \mb{C}^n \cong \mb{R}^{2n}$
over Euclidean $\mb{R}^{2n}$ space. The isometry group is generated by
the lifts of isometries of $\mb{R}^{2n}$ and translations in the vertical
direction, and has dimension $n(2n+1)+1$. There exists a neighborhood $U$
of the identity element in $H_n$. with $H_n \cap U$ connected, and such that 
any discrete subgroup $\Gamma$ of $H_n$ generated by $\Gamma \cap U$ is 
cocompact. Thus, cofinite lattice subgroups of $H_n$ are uniform, and 
a manifold with this geometry is of the form $M=H_n/\Gamma_M$ for $\Gamma_M$ 
a nilpotent 2-step discrete subgroup of $H_n$ that contains a subgroup of finite
index isomorphic to the integer Heisenberg group. Since the entire group acts 
on itself by left translations, which are isometries, $M$ is provided with
many left invariant metrics that are invariant under $\Gamma_M$, all of which
have constant scalar curvature. We thus say that $(M,\Gamma_M)$ is a 
Heisenberg nilmanifold of underlying fundamental group $\Gamma_M$, without
specifying any of these alluded $\Gamma_M$ invariant metrics on it. 

\begin{theorem}
Any $2n+1$ manifold $M^{2n+1}=H_{n\geq 1}/\Gamma_M$ of Heisenberg 
geometry is of Kazdan-Warner type {\rm III} and has $\sigma(M)=0$, but
no conformal class of metrics on $M$ achieves such a value. In particular,
all {\rm 3d} manifolds of nilgeometry are orientable of Kazdan-Warner type 
{\rm III} and vanishing nonachievable sigma invariant.  
\end{theorem}

{\it Proof}. Suppose that $M$ carries a Ricci flat metric $g$.  By
the Cheeger-Gromoll splitting theorem \cite{chgr}, its universal cover is 
isometric to the standard flat Euclidean space $\mb{R}^{2n+1}$, which. 
as a manifold. is diffeomorphic to $H_n$, and so
$\pi_1(M) \cong \Gamma_M$ must have a subgroup of finite index 
isomorphic to $\mb{Z}^{k=2n+1}$. But no 
discrete subgroup $\Gamma$ of $H_n$ can have a commutative subgroup 
of finite index isomorphic to $\mb{Z}^{k}$ for any $0 \leq k \leq 2n+1$. 
For if $\Gamma'$ is one such, then $[\Gamma',\Gamma']$ must be nontrivial and 
intersects the center in an infinite cyclic group. Thus, no such Ricci
flat metric $g$ exists. By Corollary \ref{co3}, $M$ does not carry zero 
scalar curvature metrics either, and it is a manifold of Kazdan-Warner type III.

The Lie algebra of $H_n$ contains linearly independent vectors $x,y,z$ such
$[x,y]=z$. By \cite[Theorem 3.3]{miln2}, if $b=\{b_1, \ldots, b_{2n+1}\}$ is 
a basis of the Lie algebra with $b_1=x$, $b_2=y$, $b_3=z$, if $\varepsilon>0$,
the left invariant metric that makes $b_{\varepsilon}=\{ \varepsilon b_1, 
\varepsilon b_2, \varepsilon^2 b_3, \ldots, \varepsilon^2 b_{2n+1}\}$ 
orthonormal defines a Lie algebra structure on $T_{0}H_n$ that has 
a well defined nilpotent but noncommutative limit as $\varepsilon \rightarrow
0$, provided with a metric of negative scalar curvature. By continuity, 
this scalar curvature is close to but bounded above by a negative 
constant for all $\varepsilon$ sufficiently small. If we denote 
by  $g_{b_{\varepsilon}}$ the induced metric on $M$, and by $g_b$ the metric 
corresponding to $\varepsilon =1$, we have that $\mu_{g_{b_{\varepsilon}}}(M) 
=\varepsilon^{4n}\mu_{b}(M)$. Hence, for small enough $\varepsilon$, 
$$
\lambda(M,g_{b_{\varepsilon}})=s_{g_{b_{\varepsilon}}}
(\mu_{g_{\varepsilon}}(M))^{
\frac{2}{2n+1}} =s_{g_{b_{\varepsilon}}}\varepsilon^{\frac{8n}{2n+1}}(\mu_{g_{b
}}(M))^{\frac{2}{2n+1}} < 0 \, ,
$$
and becomes arbitrarily close to $0$ as $\varepsilon \searrow 0$.

Finally, we observe that $H_n$ is naturally a contact $2n+1$ manifold, and 
any one such is orientable if $n$ is odd. In dimension 3, the only nilpotent 
non Abelian group is $H_1$. The last statement now follows.
\qed

\subsubsection{{\rm 3d} manifolds with solvgeometry}
By \cite[Theorem 3.1]{miln2}, every left invariant metric on a solvable
Lie group is either flat or else has strictly negative scalar curvature, 
and by \cite[Lemma 6.2]{miln2}, any solvable Lie group that admits uniform
lattice subgroups is unimodular. So on any manifold with solvgeometry, 
a left invariant metric is either flat or it has strictly negative 
scalar curvature, and if the group is noncommutative, it has left invariant 
metrics of negative scalar curvature. 

The 3d group $S_2$ of 
\S3.1.3 b) defines a distinguished solvmanifold geometric model that is
neither commutative nor nilpotent. The isotropy group 
${\rm Isom}_p S_2$ of any point is trivial, and so if $g$ is a left invariant
metric that makes the eigenspaces of the action of $t$ on $\mb{R}^2$ orthogonal,
the identity component of this group, which embeds into 
${\rm Isom}(S_2)$, is in fact equal to $S_2$.
We say that a 3d manifold has $S_2$ solvgeometry if 
$M=S_2/\Gamma_M$ for a cocompact lattice subgroup $\Gamma_M$ of $S_2$.  

\begin{theorem}
If $M=S_2/\Gamma_M$ is a {\rm 3d} closed manifold of $S_2$ solvgeometry and 
fundamental
group $\Gamma_M$, then $M$ is a manifold of Kazdan-Warner type {\rm III}, and
has vanishing nonachievable sigma invariant,
\end{theorem}

{\it Proof}. $M$ cannot carry a Ricci flat metric. For otherwise,
 by the Cheeger-Gromoll spliting theorem \cite{chgr}, the universal cover must
be isometric to the standard flat Euclidean space $\mb{R}^{3}$, 
and by the finite volume condition, $\pi_1(M) \cong \Gamma_M$ must have a 
subgroup of finite index isomorphic to $\mb{Z}^{3}$. 
But any lattice subgroup $\Gamma$ of $S_2$ fits into a short exact sequence
$$
0 \rightarrow \mb{Z}^2 \rightarrow \Gamma \rightarrow \mb{Z} \rightarrow 1
$$
but is never virtually Abelian. 
By Corollary \ref{co3}, $M$ does not carry zero scalar curvature metrics 
either, and $M$ is a manifold of Kazdan-Warner type III.

Although the Lie algebra $\mf{s}_2$ is not nilpotent, we can still find 
linearly independent elements $x,y,z$ such that $[x,y]=z$ to  
apply Milnor's argument in the proof of \cite[Theorem 3.3]{miln2}. 
We consider the Lie algebra basis $b=\{b_1=x, b_2=y, b_3=z\}$, and for 
$\varepsilon >0$, 
the left invariant metric $g_{b_{\varepsilon}}$ that makes of 
$b_{\varepsilon}=\{ \varepsilon b_1, \varepsilon b_2, \varepsilon^2 b_3\}$ 
an orthonormal frame. Then, 
for sufficiently small $\varepsilon$, $s_{g_{\varepsilon}}<c<0$, and 
since $\mu_{g_{b_{\varepsilon}}}(M)= \varepsilon^{4} \mu_{g_{b}}(M)$, we 
have that
$$
\lambda(M,g_{b_{\varepsilon}})=s_{g_{b_{\varepsilon}}}
(\mu_{g_{\varepsilon}}(M))^{
\frac{2}{3}} =s_{g_{b_{\varepsilon}}}\varepsilon^{\frac{8}{3}}(\mu_{g_{b
}}(M))^{\frac{2}{3}} < 0 \, ,
$$
a value that becomes arbitrarily close to $0$ as $\varepsilon \searrow 0$.
\qed

\subsubsection{{\rm 3d} manifolds with 
$\widetilde{\mb{P}\mb{S}\mb{L}}(2,\mb{R})$ geometry} \label{ama3}
Any discrete subgroup $\Gamma$ of isometries of 
$\widetilde{\mb{P}\mb{S}\mb{L}}(2,\mb{R})$ is cofinite 
if, and only if, the action on $\mb{H}^2$ of its projection onto 
${\rm Isom}(\mb{H}^2,g_{\mb{H}^2})$ is discrete, cofinite and
has infinite kernel \cite[Corollary 4.7.3]{thurs}.
It follows that a manifold with this geometric model
is of the form $M=\widetilde{\mb{P}\mb{S}\mb{L}}(2,\mb{R})/
\Gamma$ where $\Gamma$ is a cocompact discrete group of 
hyperbolic motions, and no subgroup of finite index in the quotient split, so
$M$ is an orientable Seifert fiber space over an orbifold of negative Euler 
characteristic and nonzero Euler number of the Seifert structure. Such a
a manifold is finitely covered by the unit sphere bundle
$S^1(T\Sigma)$ of a Riemann surface $(\Sigma,g)$ of genus at least two.

\begin{theorem}
Any closed manifold $(M,\Gamma_M)$ of model geometry 
$\widetilde{\mb{P}\mb{S}\mb{L}}(2,\mb{R})$ is orientable, it is a manifold
Kazdan-Warner type {\rm III}, and has vanishing nonachievable sigma invariant. 
\end{theorem}

{\it Proof}. The manifold does not carry Ricci flat metrics. For if it did,
by the Cheeger-Gromoll splitting theorem, the universal cover would be 
isometric to the standard flat Euclidean $\mb{R}^3$, and $\pi_1(M)$ would have
to have a subgroup of finite index isomorphic to $\mb{Z}^3$, 
which is not possible. By Corollary \ref{co3}, $M$ is a
manifold of Kazdan-Warner type III. 

We let $M_r'=S^r(T\Sigma)$ be the $r$-sphere bundle of a 
a uniformized surface $(\Sigma,g)$ so that $M':=M'_1=S^1(T\Sigma)$ is the unit 
sphere bundle that finitely covers $M$. The result follows if we show that
$\sigma(M')=0$. The metric $g$ on $\Sigma$ 
has constant negative scalar curvature and lies in the conformal class $[g] \in
\mc{C}(\Sigma)$. If $k$ is the genus of $\Sigma$, we have the bound 
\begin{equation} \label{ama}
0>s_g= \frac{1}{(\mc{W}(\Sigma,[g])/4)^{\frac{1}{2}}}\frac{ 8\pi \chi(\Sigma)}{
\mc{W}(\Sigma,[g])^{\frac{1}{2}}} = \frac{ 16\pi \chi(\Sigma)}{\mc{W}(\Sigma,
[g])} \geq \frac{ 4\pi \chi(\Sigma)}{\mu_{g_{\xi_{k,1}}}(\xi_{k,1})} 
=s_{\tilde{g}_{\xi_{k,1}}}\, ,
\end{equation}
where $\mu_{g_{\xi_{k,1}}}(\xi_{k,1})$ is the area of the Lawson Riemann
surface $\xi_{k,1}$ minimally embedded into $\mb{S}^3$, and whose metric
$g_{\xi_{k,1}}$ has been conformally deformed while preserving the area to a  
metric $\tilde{g}_{k,1}$ of constant scalar curvature value given by the
expression on the right above. We form the Sasaki metric $g_S$ on 
$T\Sigma$, and let $g_S^r$ be the metric on the $M_r'$s that $g_S$ induces
by restriction. By a relatively simple calculation using the fact that the 
hyperbolic metric is locally symmetric, we see that
the scalar curvature of $(M',g_S:=g_S^1)$ is    
$$
s_{g_S}=s_g-\frac{1}{2}\, ,  
$$
so $\lambda(M',g_S)= s_{g_S} (2\pi \mu_{g}(\Sigma))^{\frac{2}{3}}<0$, while
more generally, for $(M'_r,g_S^r)$, we have that
$$
\lambda(M',g^r_S)= (s_{g}-r^2/2)(2\pi r \mu_{g}(\Sigma))^{\frac{2}{3}}\, ,
$$
which tends to $0$ from below as $r\searrow 0$. 
\qed 

\subsubsection{The $\mb{R}\times {\mb H}^2$ model geometry}
Any discrete subgroup $\Gamma$ of isometries of 
$\mb{R}\times \mb{H}^2$ is cofinite if, and only if, its action on $\mb{H}^2$ 
is discrete, cofinite, and has infinite kernel 
\cite[Corollary 4.7.3]{thurs}, so if $M=\mb{R}\times \mb{H}^2/
\Gamma$ is a closed manifold with this geometry, $\Gamma$ is a cocompact 
discrete group of hyperbolic motions, and there is subgroup of finite index 
in the quotient that splits. Such an $M$ is therefore a Seifert fiber space 
over an orbifold of negative Euler characteristic and zero Euler number for 
the Seifert structure. It follows
that $M$ is a finite quotient of a trivial bundle 
$S^1 \times \Sigma$ over a Riemann surface of genus at least two.  

\begin{theorem}
Any closed manifold $(M,\Gamma_M)$ of model geometry 
$\mb{R}\times \mb{H}^2$ is a manifold of Kazdan-Warner
type {\rm III}, and has vanishing nonachievable sigma invariant. 
\end{theorem}

{\it Proof}. The manifold does not carry Ricci flat metrics. For if it did,
by the Cheeger-Gromoll splitting theorem, the universal cover would be 
isometric to the standard flat Euclidean $\mb{R}^3$, and $\pi_1(M)$ would have
to have a subgroup of finite index isomorphic to $\mb{Z}^3$,
which as in the previous case, is not possible.

We let $M'=S^1\times\Sigma$ be the trivial $\mb{S}^1$ bundle over a
Riemann surface $\Sigma$ of genus $k\geq 2$ that  
finitely covers $M$. In order to complete the proof, it suffices to show that
$\sigma(M')=0$. We provide $M'$ with the product metric $g_r'$
given by the lift of the of a uniformized metric $g$ on $\Sigma$ times the
standard metric on the circle fibers of length $2\pi r$. By 
the bound (\ref{ama}) in \S\ref{ama3}, we have that 
$$
s_{g'}=s_g =\frac{ 16\pi \chi(\Sigma)}{\mc{W}(\Sigma,
[g])} \geq \frac{ 4\pi \chi(\Sigma)}{\mu_{g_{\xi_{k,1}}}(\xi_{k,1})} 
=s_{\tilde{g}_{\xi_{k,1}}}\, ,
$$
and therefore,  
$$
s_{\tilde{g}_{\xi_{k,1}}} (2\pi r \mu_g(\Sigma))^{\frac{2}{3}}\leq 
\lambda(M',g_r')= 
\left( \frac{ 16\pi \chi(\Sigma)}{\mc{W}(\Sigma,
[g])}\right) (2\pi r \mu_g(\Sigma))^{\frac{2}{3}} 
 < 0 \, ,
$$
and therefore, $\lambda(M',g'_r) \nearrow 0$ as $r\searrow 0$, proving the
desired result.  
\qed 

\subsubsection{Hyperbolic manifolds}
A closed hyperbolic $n$ manifold is topologically the quotient
$M=\mb{H}^n/\Gamma_M$ of hyperbolic $n$-space by a cocompact 
subgroup $\Gamma_M$ of $\mb{O}^{+}(1,n)$ acting freely on it and that has no
parabolic elements. Thus, $\Gamma_M$ is a finitely generated discrete subgroup 
of ${\rm Isom}(\mb{H}^n,g_{\mb{H}^n})$ that is torsion free, so its action on 
$\mb{H}^n$ is free, and as it contains no parabolic elements, the quotient 
manifold has no cusps, and it is therefore closed. $M$ inherits a $\Gamma_M$ 
invariant metric $g_M$ of constant sectional curvature $-1$, and finite volume 
$\mu_{g_M}(M)$. 
We say that $(M,g_M,\Gamma_M)$ is a hyperbolic $n$ manifold of 
underlying fundamental group $\Gamma_M$. 

A discrete subgroup $\Gamma$ of $\mb{O}^{+}(1,n)$ is called a Kleinian group,
and its action on $\mb{H}^n$ is properly discontinuous. If $\Gamma$ is torsion 
free, $M_{\Gamma}:=\mb{H}^n/\Gamma$ is a manifold that if  
$\Gamma$ is in addition cocompact, is compact, so of finite volume, and 
$\Gamma\cong \pi_1(M_{\Gamma})$ must be then finitely generated. If this 
$\Gamma$  contains no parabolic elements, 
there are no cusps, every infinite order element of $\Gamma$ is loxodromic, 
and $M_{\Gamma}$ is closed.  Such a group admits a convex
fundamental polyhedron with finitely many faces. 

\begin{theorem}
Let $(M^n,g_{M},\Gamma_M)$ be a hyperbolic manifold of underlying 
fundamental group $\Gamma_M$. Then $M^{n\geq 2}$ is a manifold of 
Kazdan-Warner type {\rm III}. If $n\geq 3$, the hyperbolic metric $g_M$ and 
its conformal class realize 
$$
\sigma(M)= 
-n(n-1) \mu_{g_M}(M)^{\frac{2}{n}}\, ,   
$$
any class $[g] \in \mc{C}(M)$ admits a metric of constant negative  
sectional curvature representative, unique up to a homothetic dilation, and 
the space of hyperbolic metrics on $M$ is path connected and consists of 
isotopic deformations of $g_M$ by equal volume hyperbolic metrics that are
all isometric to each other.
\end{theorem}

{\it Proof}. Since $g_M$ has constant sectional curvature $-1$, $n\geq 2$.  
If $n=2$, $\chi(M)<0$ and using the extended notion we had introduced,
 $M$ is a surface of 
Kazdan-Warner type III.  We thus proceed with the proof assuming that $n\geq 3$.

Since $g_M$ is an Einstein metric of scalar curvature $-n(n-1)$, by the 
uniqueness of solution of the Yamabe equation in the nonpositive case, $g_M$ 
is the unique Yamabe metric of volume $\mu_{g_M}(M)$ in its conformal class
$[g_M]$. We take an arbitrary metric $g$ in $\mc{M}(M)$, let 
$[0,1]\ni t \rightarrow g_t$ be any smooth path of metrics connecting 
$g_{M}$ to $g$, and consider the lift of the path of conformal classes 
$t \rightarrow [g_t]$ to a path $t\rightarrow g_t^Y$ of Yamabe metrics in 
$[g_t]$ such that $\mu_{g_t}=\mu_{g^Y_t}$, connecting $g_{M}$ to a Yamabe
metric $g_t^Y\mid_{t=1}$ of volume $\mu_g$, in the conformal class of $g$. 
We consider the associated path $[0,1]\ni t \rightarrow f_{g_t^Y}$ of
isometric embeddings of the $g_t^Y$s.  

It is then the case that $s_{g_t^Y}$ is a negative constant for all
$t \in [0,1]$. For otherwise, there would exists 
some $\bar{t}\in (0,1]$ such that  $[0,\bar{t})\ni t \rightarrow s_{g_t^{Y}}<0$ and $s_{g_{\bar{t}}^{Y}}=0$, in which case, we look at the restricted paths
of Yamabe metrics $[0,\bar{t}] \ni t \rightarrow g_t^Y$ 
and associated path of isometric (differentiable) embeddings 
$[0,\bar{t}] \ni t \rightarrow f_{g_t^Y}$. 
By the negativity of the path of scalar curvatures, they satisfy the
hypotheses of Theorem \ref{am}, so the former is a path of constant
sectional curvature metrics, and the sign of these curvatures along the
entire path is unique, in contradiction with the fact that the sign of the
sectional curvatures of the metric $g_M=g_t^Y\mid_{t=0}$ is negative, while 
that of the curvatures of $g_t^Y\mid_{t=\bar{t}}$ must be zero.   
By the arbitrary choice of $g$, it follows that $M$ 
does not carry any scalar flat metric, so $M$ is a manifold of 
Kazdan-Warner type III, and that the entire paths   
$[0, 1] \ni t \rightarrow g_t^Y$ and 
$[0, 1] \ni t \rightarrow f_{g_t^Y}$ satisfy the hypotheses 
of Theorem \ref{am}. 
Therefore, $[0,1]\ni t \rightarrow g_t^Y$ is a path of Einstein metrics of 
constant sectional curvature each, and
$$
\lambda(M,g_t^Y)= s_{g_t^Y} \mu_{g_t^Y}(M)^{\frac{2}{n}}=
-n(n-1) \mu_{g_M}(M)^{\frac{2}{n}}= \lambda(M,g_M):=\sigma(M) \, .
$$
Hence, no matter what the metric $g$ is, the Yamabe representative 
$g^Y_t\mid_{t=1}$ of $[g]$ is an Einstein metric of constant negative 
sectional curvature that realizes $\sigma(M)$.

Suppose now that in the argument above, $g$ is a hyperbolic metric on $M$
such that $\mu_{g}=\mu_{g_M}=:\mu$. If $[0,1]\ni t \rightarrow g_t$ is an
arbitrary path of volume $\mu$ metrics connecting $g_M$ and $g$, 
its lift $t \rightarrow g_t^Y$ by volume $\mu$ Yamabe metrics in $[g_t]$ 
are all of the same constant sectional curvature $-1$, 
that is to say, hyperbolic. By Mostow rigidity theorem \cite{mos}, 
$(M,g_t^Y)$ is 
isometric to $(M,g_M)$.   
\qed

\section{The $\sigma^J$ invariant of some $4$d almost Hermitian manifolds} 
Suppose that $M^{n=2m}$ carries an almost complex structure $J_0$. We let
$\mc{J}_{J_0}(M)$ be the path connected manifold of all almost complex 
structures on $M$ in the same orientation class as $J_0$, and denote by  
$\mc{I}_{J_0}(M,g)$ the space of all almost complex structures in 
$\mc{J}_{J_0}(M)$ that are compatible with a given Riemannian metric $g$.
If $\mc{M}^{J_0}(M)$ is the space of 
metrics on $M$ that are compatible with at least one element of 
$\mc{J}_{J_0}(M)$, and $\mc{C}^{J_0}(M)$ the space of conformal classes 
of such, we have a fibration
\begin{equation} \label{eq20}
\mc{M}^{J_0}(M) \stackrel{\pi}{\rightarrow } \mc{C}^{J_0}(M) \, .    
\end{equation}
We set $\mc{M}^{J_0}_{[g]}(M)$ to be the fiber over   
the conformal class $[g]$ of $g \in \mc{M}^{J_0}(M)$. We have 
the foliated decomposition $\mc{M}^{J_0}(M)=\cup_{[g]\in
 \mc{C}^{J_0}(M)}\mc{M}^{J_0}_{[g]}(M)$. 
Any metric $g$ in $[g]\in \mc{C}^{J_0}(M)$ is invariant under any element of 
$\mc{I}_{J_0}(M,g)$. Notice that along a smooth path 
$t \rightarrow (J_t,g_t)$ of metrics in $\mc{M}^{J_0}(M)$, the Chern classes
of $(TM,J_t)$ are fixed and equal those of $(TM,J_t\mid_{t=0})$, although
their Levi-Civita connection $\nabla^{g_t}$ representatives may change but
vary smoothly with $g_t$.      

If $g\in \mc{M}^{J_0}(M)$, and $J\in \mc{I}_{J_0}(M,g)$, we define the
tensors 
$$
r_{g}^J(X,Y)={\rm trace}\, L \rightarrow -J(R^{g}(L,X)JY) \, , 
$$
and 
$$
s_g^J= {\rm trace}\, r^J_g \, , 
$$
respectively. If $n=2m=2$, $(J,g)$ is a K\"ahler structure on $M$, 
 $r_g^J=r_g$ and $s^J_g=s_g$. We assume otherwise, and have that
$$
r^J= -\frac{s_g}{2(m-1)(2m-1)}g+\frac{2}{m-1}r_g^{1,1}-\sum_i W_g(e_i,\cdot,
J\cdot, Je_i) \, ,
$$
where $W_g$ is the Weyl tensor of $g$, and $r_g^{1,1}$ is the $J$-invariant
component of the Ricci tensor of $g$. Hence, the tensor $r_g^J$ may fail to be
symmetric if, and only if, the contribution above from the Weyl tensor fails 
to be so, its skew Hermitian component is a conformal invariant (that may vary 
with $J \in \mc{I}_{J_0}(M,g)$) \cite[Corollary 4.5]{rss}, and 
\begin{equation} \label{rem} 
(n-1)s_g^J-s_g=2(n-1)W_g(\omega_g^{\sharp},\omega_g^{\sharp})\, , 
\end{equation}
where $\omega_g^{\sharp}$ is the fundamental form viewed as a 
bivector \cite[Proposition 4.2]{rss}. 
Notice that by definition, $r^J_g$ depends on two covariant derivatives of 
$J$, yet
the dependence of $s^J_g$ on $J$ is tensorial of order zero.
Since all symplectic vector spaces of a 
given dimension are isomorphic, $s_g^J$ is a function of $\mc{I}_{J_0}(M,g)$ but
not of the particular choice of $J$ chosen to define it, and the 
scale invariant functional  
\begin{equation} \label{jyfu}
\mc{M}^{J_0}_{[g]}(M) \ni g \rightarrow \lambda^{J}(M,g) = 
\frac{1}{\mu_g^{\frac{2}{N}} }\int s_g^J d\mu_{g} 
\end{equation}
is well-defined. 

If $g_t = e^{2u(t)}g$ is a path of metrics in the conformal class of the $J$
invariant metric $g$, 
$g'=g_t{\mid_{t=1}}$ is a critical point of 
{\rm (\ref{jyfu})} if, and only if, $s^J_{g'}$ is a constant satisfying the
equation  
\begin{equation} \label{gr2}
s^J_{g'}= e^{-2u}\left( \frac{1}{n-1}s_g +2
W_{g}(\omega^{\sharp}_g, \omega^{\sharp}_g)  
+ 2 \Delta^{g} u^{\tau} -(n-2)\tg(du^{\tau},du^{\tau}) 
\right)_{\mid_{t=1}} 
= \frac{1}{\mu_{g'}}\int_{f_{g'}(M)} \hspace{-0.1cm} s^J_{g'} d\mu_{g'}\, . 
\end{equation}
Such a metric has constant scalar curvature $s_{g'}$ also if, and only if, 
$W_{g'}(\omega^{\sharp}_{g'}, \omega^{\sharp}_{g'})$ is a constant function,
in which case, $(n-1)\lambda^J(M,g')=\lambda(M,g')+ 2(n-1)W_{g'}( 
\omega^{\sharp}_{ g'}, \omega^{\sharp}_{g'})\mu_{g'}^{\frac{2}{n}}$.
Metrics that realize the intrinsic conformal invariant 
$$ 
\lambda^{J}(M,[g])= \inf_{g\in \mc{M}^{J_0}_{[g]}(M)} 
\frac{1}{\mu_g(M)^{\frac{2}{N}}}\int s^J_g d\mu_g  
$$ 
exists \cite{rss}, and are by definition the $J$ Yamabe metric representatives 
of $[g]$. By Aubin's \cite[Theorem p. 155]{au2}, we have the universal bound
\begin{equation} \label{univ}
\lambda^{J}(M,[g])\leq 2m \omega_{2m}^{\frac{1}{m}} \, 
= n \omega_{n}^{\frac{2}{n}} \, ,
\end{equation}
and since the left side varies continuously as a function of 
$(J,[g])$ \cite[ Lemma 2]{sim5}, 
$$ 
\sigma^{J_0}(M)= \sup_{[g]\in \mc{C}^{J_0}(M)} \lambda^J(M,[g]) 
$$ 
is a well-defined invariant of the differentiable almost Hermitian 
manifold $(M,J_0)$. 

This notion extends to oriented surfaces. Indeed, 
$(J,g)$ is then a  K\"ahler structure, and 
the total scalar curvature is the topological invariant 
$4\pi \chi(M)=8\pi (1-k)$, hence it is natural to set 
$\sigma^J(M):=\sigma(M)= 
8\pi (1-k)/(\mc{W}(M,[g_{\xi_{k,1}}]/4)^{\frac{1}{2}})$, where as 
indicated earlier, the optimal conformal class $[\xi_{k,1}]$ on the surface 
is the minimizer of $[g]\rightarrow \mc{W}(M,[g])$ among conformal 
classes on the Riemann surface $M$ of genus $k$.

We let $(T^4_{\Lambda_{\Box}}, J_{\Lambda_{\Box}})$ be the complex 2-torus of
square lattice $\Lambda_{\Box} \subset \mb{C}^2$ with generators $(1,0)$, 
$(i,0)$, $(0,1)$ $(0,i)$, and let  
$(X_{T_{\Lambda_{\Box}}},J_{\Lambda_{\Box}})$ be the K3 Kummer
surface of underlying torus $T^4_{\Lambda_{\Box}}$. The underlying real 
manifolds of these two are both of Kazdan-Warner type II, 
and the canonical realizers of their vanishing sigma invariants are the 
conformal classes of the standard metric $g_{\Lambda_{\Box}}$ associated to the
square lattice minimally embedded into $\mb{S}^{7}$, and 
the Yau Ricci flat metric $g^Y_{s_{\Box}}$ on 
$(X_{T_{\Lambda_{\Box}}},J_{\Lambda_{\Box}})$, respectively. These metrics
provide these manifolds with a canonical volume \cite[\S3.1 \& \S3.2]{sim7}. 
We let $(\mb{S}^1\times
\mb{S}^3, J_0)$ be the total space of the
Calabi-Eckmann bundle over $\mb{P}^1(\mb{C})$ of flat tori fibers. This
is a manifold of Kazdan-Warner type I, which does not carry Einstein metrics
of positive scalar curvature.

\begin{theorem}
\begin{enumerate}[label={\rm (\alph*)}]
\item We have that $\sigma^{J_{\Lambda_{\Box}}}(T^4_{\Lambda_{\Box}})=0$, 
and among all the metrics $g$ that achieve this value and have minimal
isometric embedding $f_g$, the flat
K\"ahler metric $g_{\Lambda_{\Box}}$ has the smallest volume $\pi^4$.   
\item We have that $\sigma^{J_{\Lambda_{\Box}}}(X_{T^4_{\Lambda_{\Box}}})=0$, 
and among all the metrics $g$ that achieve this value and have minimal
isometric embedding $f_{g}$,  the Yau Ricci flat K\"ahler metric 
$g^Y_{s_{\Lambda_{\Box}}}$ has the smallest volume $2\pi^2$.   
\item We have that $\sigma^{J_0}(\mb{S}^1\times \mb{S}^3)=4\omega_4^{
\frac{1}{2}}$, the universal bound {\rm (\ref{univ})}, 
but no conformal class of almost Hermitian metrics achieve this value.  
\end{enumerate}
\end{theorem}

{\it Proof}. (a) The first Chern class of any complex 2 torus
is zero, and any Ricci flat metric $g$ that is Hermitian with respect to
a deformation of $J_{\Lambda_{\Box}}$ is such that 
$\lambda^{J_{\Lambda_{\Box}}}(T^4_{\Lambda_{\Box}},[g]))=
\lambda^{J_{\Lambda_{\Box}}}(T^4_{\Lambda_{\Box}},g)=0$, with 
$W^{g}(\omega_g^{\sharp},\omega_g^{\sharp})=0$. Among these, 
$g_{\Lambda_{\Box}}$ is the optimal realizer, with
$\mu_{g_{\Lambda_{\Box}}}(T^4_{\Lambda_{\Box}})=\pi^4$. We finish the proof
by showing that for no almost Hermitian structure $(g,J)$ on this manifold it
is the case that $\lambda^{J}(M,g)>0$. 

Suppose that $(J,g)$ is any other almost Hermitian $J$ Yamabe metric on 
this torus. We consider a path $[0,1] \ni t \rightarrow (J_t,g_t)$ of almost
Hermitian structures that connects $(J_{\Lambda_{\Box}}, g_{\Lambda_{\Box}})$ 
with it, and let $t\rightarrow g_t^Y$ be the path of Yamabe metrics in $[g_t]$ 
of volume $\mu_{g_t}(T^4)$. Since $T^4$ is a manifold of Kazdan-Warner 
type II, $s_{g_t^Y} \leq 0$, and the path of Yamabe metrics is smooth, with
$g_t^Y\mid_{t=0}$ a metric of constant sectional curvature $0$.
By Theorem \ref{am}, the sectional curvature of $g_t^Y$ remains the constant
$0$ for all $t\in [0,1]$. Hence, $g^Y:=g^Y_t\mid_{t=1}$ is flat, and
so $W^{g^Y}(\omega_{g^Y}^{\sharp},\omega_{g^Y}^{\sharp})=0$. 
Since the $J$ Yamabe metric $(J,g)$ is a conformal deformation of
$g^Y$ (so (\ref{gr2}) holds), it follows that
$\lambda^{J_{\Lambda_{\Box}}}(T^4_{\Lambda_{\Box}},[g]))=
\lambda^{J_{\Lambda_{\Box}}}(T^4_{\Lambda_{\Box}},g)=0$. Thus, 
$\lambda^{J_{\Lambda_{\Box}}}(T^4_{\Lambda_{\Box}})=0$, as desired.

(b) The first Chern class of any homotopic deformation of $(X_{T_{\Lambda_{\Box}
}}, J_{\Lambda_{\Box}})$ is zero, and any Ricci flat metric that is
Hermitian with respect to a deformation of $J_{\Lambda_{\Box}}$ is such that 
$\lambda^{J_{\Lambda_{\Box}}}(X_{T_{\Lambda_{\Box}}},[g]))=
\lambda^{J_{\Lambda_{\Box}}}(X_{T_{\Lambda_{\Box}}},g)=0$, with 
$W^{g}(\omega_g^{\sharp},\omega_g^{\sharp})=0$. Among these, the Yau 
Ricci flat K\"ahler metric $g^Y_{s_{\Box}}$ is the optimal realizer of the 
sigma invariant of the differentiable underlying manifold, with
$\mu_{g^Y_{s_{\Box}}}(X_{T_{\Lambda_{\Box}}})=2\pi^2$. 

Suppose that $(J,g)$ is any other almost Hermitian $J$ Yamabe metric on 
$(X_{T_{\Lambda_{\Box}}},J_{\Lambda})$, and consider a path 
$[0,1] \ni t \rightarrow (J_t,g_t)$ of almost Hermitian structures that 
connects $(J_{\Lambda_{\Box}}, g^Y_{s_{\Box}})$ with it, and the associated
lift $t\rightarrow g_t^Y$ of Yamabe metrics in $[g_t]$ 
of volume $\mu_{g_t}(X_{T_{\Lambda_{\Box}}})$. Since the surface is a manifold
of Kazdan-Warner type II, $s_{g_t^Y} \leq 0$, so the path of Yamabe metrics is 
smooth, though we cannot ascertain now that even $g_{s_{\Box}}^Y=
g_t^Y\mid_{t=0}$ has zero 
sectional curvature, which is only so away from the 16 
singular points of the underlying torus orbifold 
$T_{\Lambda_{\Box}}/(\mb{Z}/2)$ whose desingularization (replacing the
points by $-2$ curves) defines the surface. However, there is an associated
path $t \rightarrow (J_t, \tilde{g}_t)$ of almost Hermitian
metrics on the 2-to-1 covering torus $T_{\Lambda_{\Box}}$ of the orbifold 
that corresponds to $t \rightarrow (J_t,g_t)$. We let
$t \rightarrow \tilde{g}^Y_t$ be the associated path of Yamabe metrics. 
Since $\tilde{g}_t^Y\mid_{t=0}$ has zero sectional curvature, proceeding as 
in (a), $\tilde{g}^Y_t$ has zero sectional curvature for all $t$, so
$W^{\tilde{g}_t^Y}( \omega_{\tilde{g}^Y_t}^{\sharp}, 
\omega_{\tilde{g}^Y_t}^{\sharp})=0$.
By smoothness and the codimension of the $-2$ curves, this last condition 
passes to the metrics $g_t^Y$ themselves, so 
$[0,1]\ni t \rightarrow W^{g_t^Y}( \omega_{g^Y_t}^{\sharp}, 
\omega_{g^Y_t}^{\sharp})=0$. This fact used at $t=1$ implies that 
$\lambda^{J_{\Lambda_{\Box}}}(X_{T_{\Lambda_{\Box}}},[g]))=
\lambda(X_{T_{\Lambda_{\Box}}},g)\leq 0$. Thus, the $J$ sigma invariant of
any almost Hermitian conformal class on $X_{T_{\Lambda_{\Box}}}$ 
is at most zero, and  
$\lambda^{J_{\Lambda_{\Box}}}(X_{T_{\Lambda_{\Box}}})=0$, as desired.

(c) We let $\{ g_n\}$ be Schoen's sequence of Yamabe metrics in a sequence of
product conformal classes in $\mb{S}^1 \times \mb{S}^3$ \cite{sc2}. 
Since the product metrics are all locally conformally flat, 
$W^{g_n}(\omega_{g_n}^{\sharp}, \omega_{g_n}^{\sharp})=0$, and 
$3\lambda^{J_0}(\mb{S}^1\times        
\mb{S}^3,[g_n])=\lambda(\mb{S}^1\times \mb{S}^3,[g_n]) \nearrow 12
 \omega_4^{\frac{1}{2}}$. 
\qed

\end{document}